\documentclass[10pt,oneside,reqno]{amsart}

\pdfoutput=1
\usepackage{accents}
\usepackage{afterpage}
\usepackage{amscd}
\usepackage{amsfonts}
\usepackage{amsmath}
\usepackage{amssymb}
\usepackage{amsthm}
\usepackage{appendix}

\usepackage{bbm}
\usepackage{bm}
\usepackage{braket}

\usepackage{calc}
\usepackage{caption}
\usepackage{changepage}
\usepackage{color}
\usepackage{comment}

\usepackage{dsfont}

\usepackage{enumitem}
\usepackage{etoolbox}
\usepackage{extpfeil}

\usepackage{graphicx}
\usepackage[margin=1in]{geometry}

\usepackage{import}

\usepackage{mathbbol}
\usepackage{mathrsfs}
\usepackage{mathtools}

\usepackage{pifont}

\usepackage{physics}

\usepackage{relsize}

\usepackage{scalerel}
\usepackage{stackengine}
\usepackage{stackrel}
\usepackage{stmaryrd}
\usepackage{subcaption}

\usepackage{tikz}
    \usetikzlibrary{arrows.meta}
    \usetikzlibrary{shapes}
    \usetikzlibrary{decorations.shapes}
    \usetikzlibrary{decorations.pathmorphing}

\usepackage{tikz-cd}
\usepackage{transparent}

\usepackage{wasysym}
\usepackage{wrapfig}

\usepackage{xcolor}
\usepackage{xfrac}
\usepackage{xparse}

\makeatletter
\let\old@tocline\@tocline
\let\section@tocline\@tocline
\newcommand{\subsection@dotsep}{4.5}
\newcommand{\subsubsection@dotsep}{4.5}
\patchcmd{\@tocline}
  {\hfil}
  {\nobreak
     \leaders\hbox{$\m@th
        \mkern \subsection@dotsep mu\hbox{.}\mkern \subsection@dotsep mu$}\hfill
     \nobreak}{}{}
\let\subsection@tocline\@tocline
\let\@tocline\old@tocline

\patchcmd{\@tocline}
  {\hfil}
  {\nobreak
     \leaders\hbox{$\m@th
        \mkern \subsubsection@dotsep mu\hbox{.}\mkern \subsubsection@dotsep mu$}\hfill
     \nobreak}{}{}
\let\subsubsection@tocline\@tocline
\let\@tocline\old@tocline

\let\old@l@subsection\l@subsection
\let\old@l@subsubsection\l@subsubsection

\def\@tocwriteb#1#2#3{%
  \begingroup
    \@xp\def\csname #2@tocline\endcsname##1##2##3##4##5##6{%
      \ifnum##1>\c@tocdepth
      \else \sbox\z@{##5\let\indentlabel\@tochangmeasure##6}\fi}%
    \csname l@#2\endcsname{#1{\csname#2name\endcsname}{\@secnumber}{}}%
  \endgroup
  \addcontentsline{toc}{#2}%
    {\protect#1{\csname#2name\endcsname}{\@secnumber}{#3}}}%

\newlength{\@tocsectionindent}
\newlength{\@tocsubsectionindent}
\newlength{\@tocsubsubsectionindent}
\newlength{\@tocsectionnumwidth}
\newlength{\@tocsubsectionnumwidth}
\newlength{\@tocsubsubsectionnumwidth}
\newcommand{\settocsectionnumwidth}[1]{\setlength{\@tocsectionnumwidth}{#1}}
\newcommand{\settocsubsectionnumwidth}[1]{\setlength{\@tocsubsectionnumwidth}{#1}}
\newcommand{\settocsubsubsectionnumwidth}[1]{\setlength{\@tocsubsubsectionnumwidth}{#1}}
\newcommand{\settocsectionindent}[1]{\setlength{\@tocsectionindent}{#1}}
\newcommand{\settocsubsectionindent}[1]{\setlength{\@tocsubsectionindent}{#1}}
\newcommand{\settocsubsubsectionindent}[1]{\setlength{\@tocsubsubsectionindent}{#1}}

\renewcommand{\l@section}{\section@tocline{1}{\@tocsectionvskip}{\@tocsectionindent}{}{\@tocsectionformat}}%
\renewcommand{\l@subsection}{\subsection@tocline{2}{\@tocsubsectionvskip}{\@tocsubsectionindent}{}{\@tocsubsectionformat}}%
\renewcommand{\l@subsubsection}{\subsubsection@tocline{3}{\@tocsubsubsectionvskip}{\@tocsubsubsectionindent}{}{\@tocsubsubsectionformat}}%
\newcommand{\@tocsectionformat}{}
\newcommand{\@tocsubsectionformat}{}
\newcommand{\@tocsubsubsectionformat}{}
\expandafter\def\csname toc@1format\endcsname{\@tocsectionformat}
\expandafter\def\csname toc@2format\endcsname{\@tocsubsectionformat}
\expandafter\def\csname toc@3format\endcsname{\@tocsubsubsectionformat}
\newcommand{\settocsectionformat}[1]{\renewcommand{\@tocsectionformat}{#1}}
\newcommand{\settocsubsectionformat}[1]{\renewcommand{\@tocsubsectionformat}{#1}}
\newcommand{\settocsubsubsectionformat}[1]{\renewcommand{\@tocsubsubsectionformat}{#1}}
\newlength{\@tocsectionvskip}
\newcommand{\settocsectionvskip}[1]{\setlength{\@tocsectionvskip}{#1}}
\newlength{\@tocsubsectionvskip}
\newcommand{\settocsubsectionvskip}[1]{\setlength{\@tocsubsectionvskip}{#1}}
\newlength{\@tocsubsubsectionvskip}
\newcommand{\settocsubsubsectionvskip}[1]{\setlength{\@tocsubsubsectionvskip}{#1}}

\patchcmd{\tocsection}{\indentlabel}{\makebox[\@tocsectionnumwidth][l]}{}{}
\patchcmd{\tocsubsection}{\indentlabel}{\makebox[\@tocsubsectionnumwidth][l]}{}{}
\patchcmd{\tocsubsubsection}{\indentlabel}{\makebox[\@tocsubsubsectionnumwidth][l]}{}{}

\newcommand{\@sectypepnumformat}{}
\renewcommand{\contentsline}[1]{%
  \expandafter\let\expandafter\@sectypepnumformat\csname @toc#1pnumformat\endcsname%
  \csname l@#1\endcsname}
\newcommand{\@tocsectionpnumformat}{}
\newcommand{\@tocsubsectionpnumformat}{}
\newcommand{\@tocsubsubsectionpnumformat}{}
\newcommand{\setsectionpnumformat}[1]{\renewcommand{\@tocsectionpnumformat}{#1}}
\newcommand{\setsubsectionpnumformat}[1]{\renewcommand{\@tocsubsectionpnumformat}{#1}}
\newcommand{\setsubsubsectionpnumformat}[1]{\renewcommand{\@tocsubsubsectionpnumformat}{#1}}
\renewcommand{\@tocpagenum}[1]{%
  \hfill {\mdseries\@sectypepnumformat #1}}

\let\oldappendix\appendix
\renewcommand{\appendix}{%
  \leavevmode\oldappendix%
  \addtocontents{toc}{%
    \protect\settowidth{\protect\@tocsectionnumwidth}{\protect\@tocsectionformat\sectionname\space}%
    \protect\addtolength{\protect\@tocsectionnumwidth}{2em}}%
}
\makeatother

\makeatletter
\settocsectionnumwidth{2em}
\settocsubsectionnumwidth{2.5em}
\settocsubsubsectionnumwidth{3em}
\settocsectionindent{1pc}%
\settocsubsectionindent{\dimexpr\@tocsectionindent+\@tocsectionnumwidth}%
\settocsubsubsectionindent{\dimexpr\@tocsubsectionindent+\@tocsubsectionnumwidth}%
\makeatother

\settocsectionvskip{10pt}
\settocsubsectionvskip{3pt}
\settocsubsubsectionvskip{0pt}
    
\settocsectionformat{\bfseries}
\settocsubsectionformat{\mdseries}
\settocsubsubsectionformat{\mdseries}
\setsectionpnumformat{\bfseries}
\setsubsectionpnumformat{\mdseries}
\setsubsubsectionpnumformat{\mdseries}

\let\oldtableofcontents\tableofcontents
\renewcommand{\tableofcontents}{%
  \vspace*{-\linespacing}
  \oldtableofcontents}

\DeclareSymbolFontAlphabet{\amsmathbb}{AMSb}

\DeclareFontFamily{U}{stix2bb}{\skewchar\font127 }
\DeclareFontShape{U}{stix2bb}{m}{n} {<-> stix2-mathbb}{}
\DeclareMathAlphabet{\mathbbnew}{U}{stix2bb}{m}{n}

\usepackage[
backref=true,
backrefstyle=none,
backend=biber,
style=alphabetic,
citestyle=alphabetic,
maxnames=50,
maxalphanames=7
]{biblatex}

\DefineBibliographyStrings{english}{
  backrefpage = {cited on page},
  backrefpages = {cited on pages},
}

\usepackage[linktocpage=true]{hyperref}

\mathtoolsset{showonlyrefs,showmanualtags}

\allowdisplaybreaks

\newtheorem{theorem}{Theorem}
\newtheorem{proposition}[theorem]{Proposition}
\newtheorem{conjecture}[theorem]{Conjecture}
\newtheorem{corollary}[theorem]{Corollary}
\newtheorem{lemma}[theorem]{Lemma}

\theoremstyle{definition}
\newtheorem{definition}[theorem]{Definition}
\newtheorem{remark}[theorem]{Remark}
\newtheorem{example}[theorem]{Example}

\numberwithin{theorem}{section}
\numberwithin{equation}{section}
\DeclareDocumentCommand{\faktor}{s m O{0.5} m O{-0.5}}{
  \setbox0=\hbox{\ensuremath{#2}}
  \setbox1=\hbox{\ensuremath{\diagup}}
  \setbox2=\hbox{\ensuremath{#4}}
  \raisebox{#3\ht1}{\usebox0}
  \mkern-5mu
    {\rotatebox{-44}{\rule[#5\ht2]{0.4pt}{-#5\ht2+#3\ht0+\ht0}}}
  \mkern-4mu%
  \raisebox{#5\ht2}{\usebox2}
}
\DeclareDocumentCommand{\doublefaktor}{s m O{0.5} m O{-0.5}}{
  \setbox0=\hbox{\ensuremath{#2}}
  \setbox1=\hbox{\ensuremath{\diagup}}
  \setbox2=\hbox{\ensuremath{#4}}
  \raisebox{#3\ht1}{\usebox0}
  \mkern-5mu
  {\rotatebox{-44}{\rule[#5\ht2]{0.4pt}{-#5\ht2+#3\ht0+\ht0}}}
  \mkern-14mu
  {\rotatebox{-44}{\rule[#5\ht2]{0.4pt}{-#5\ht2+#3\ht0+\ht0}}}
  \mkern-4mu
  \raisebox{#5\ht2}{\usebox2}
}
\makeatletter
\newcommand{\pushright}[1]{\ifmeasuring@#1\else\omit\hfill$\displaystyle#1$\fi\ignorespaces}
\newcommand{\pushleft}[1]{\ifmeasuring@#1\else\omit$\displaystyle#1$\hfill\fi\ignorespaces}
\makeatother
\input pdf-trans
\newbox\qbox
\def\usecolor#1{\csname\string\color@#1\endcsname\space}
\newcommand\bordercolor[1]{\colsplit{1}{#1}}
\newcommand\fillcolor[1]{\colsplit{0}{#1}}
\newcommand\colsplit[2]{\colorlet{tmpcolor}{#2}\edef\tmp{\usecolor{tmpcolor}}%
  \def\tmpB{}\expandafter\colsplithelp\tmp\relax%
  \ifnum0=#1\relax\edef\fillcol{\tmpB}\else\edef\bordercol{\tmpC}\fi}
\def\colsplithelp#1#2 #3\relax{%
  \edef\tmpB{\tmpB#1#2 }%
  \ifnum `#1>`9\relax\def\tmpC{#3}\else\colsplithelp#3\relax\fi
}
\newcommand\outline[1]{\leavevmode%
  \def\maltext{#1}%
  \setbox\qbox=\hbox{\maltext}%
  \boxgs{Q q 2 Tr \thickness\space w \fillcol\space \bordercol\space}{}%
  \copy\qbox%
}
\newcommand\mathcalbb[2][1]{%
   \stackengine{0pt}{\outline{$\mathcal{#2}$}}{\kern.3pt\outline{$\mathcal{#2}$}}{O}{l}{F}{F}{L}}
\bordercolor{black}
\fillcolor{white}
\def\thickness{0.35}
\makeatletter
\newcommand{\leqnomode}{\tagsleft@true\let\veqno\@@leqno}
\makeatother

\newcommand{\orep}{\mathcal{O}(\operatorname{Rep}_N(A))}
\newcommand{\rep}{\operatorname{Rep}_N(A)}
\newcommand{\GL}{\operatorname{GL}_N}

\newcommand{\s}{\operatorname{S}(A_{\natural})}
\renewcommand{\S}{\amsmathbb{S}}

\renewcommand{\O}{\mathcalbb{O}}

\newcommand{\Jac}{\mathbbnew{Jac}}
\renewcommand{\tr}{\operatorname{tr}_N}
\renewcommand{\a}{\alpha}
\renewcommand{\b}{\beta}
\newcommand{\cc}{\gamma}

\newcommand{\id}{\operatorname{id}}

\DeclareMathSymbol{\mh}{\mathord}{operators}{`\-}

\newcommand{\Der}{\operatorname{Der}}

\newcommand{\br}{\{\!\!\{-,-\}\!\!\}}
\newcommand{\T}{\operatorname{T}}
\newcommand{\Assoc}{\operatorname{Assoc}}
\newcommand{\dbr}[2]{\{\!\!\{#1,#2\}\!\!\}}
\newcommand{\kk}{\Bbbk}
\newcommand{\RR}{\mathcal{R}}
\definecolor{blue-violet}{rgb}{0.54, 0.17, 0.89}
\definecolor{darkpastelgreen}{rgb}{0.01, 0.75, 0.24}
\definecolor{darkpowderblue}{rgb}{0.0, 0.2, 0.6}
\definecolor{dodgerblue}{rgb}{0.12, 0.56, 1.0}
\definecolor{denim}{rgb}{0.08, 0.38, 0.74}
\definecolor{amber}{rgb}{1.0, 0.75, 0.0}
\definecolor{applegreen}{rgb}{0.55, 0.71, 0.0}

\title{Coupled double Poisson brackets}

\author{Nikita Safonkin}

\address{Institute of Mathematics, Leipzig University, Augustusplatz 10, 04109 Leipzig, Germany.}

\email{safonkin.nik@gmail.com}

\begin{document}

\begin{abstract}
We introduce \textit{coupled double Poisson brackets}
on an associative algebra $A$ as pairs consisting of a generalized Van den Bergh's double Poisson
bracket and a generalized Fairon--McCulloch's
right double Poisson bracket subject to a
cross-Jacobi identity. Each of Van den Bergh's double brackets,
Fairon--McCulloch's right double brackets, and also Ginzburg--Schedler's
wheeled Poisson brackets induces a $\operatorname{GL}_N$-invariant Poisson structure on the representation scheme $\operatorname{Rep}_N(A)$ parametrizing $N$-dimensional representations of $A$, thereby satisfying the Kontsevich--Rosenberg
principle. Wheeled Poisson brackets seem to be the most general such
structures, and while their relation to Van den Bergh's double Poisson
brackets is known,
their relation to Fairon--McCulloch's
right double Poisson brackets has remained open. We fill this gap and establish a
bijection between pairs of coupled double Poisson
brackets and wheeled Poisson brackets of
Ginzburg and Schedler. On free polynomial algebras,
we furthermore establish a one-to-one correspondence
between linear coupled double Poisson brackets and a
new algebraic structure that we call
\textit{Poisson-left-pre-Lie algebras}, and describe
quadratic ones via solutions of the associative and
classical Yang--Baxter equations satisfying a compatibility condition.
\end{abstract}

\maketitle

    \setcounter{tocdepth}{2}
    \tableofcontents
    
    \section{Introduction}

A guiding principle in noncommutative geometry, usually referred to as the Kontsevich--Rosenberg principle \cite{kontsevich2000noncommutative}, \cite{kontsevich1993formal}, suggests that geometric structures on a noncommutative algebra $A$ should induce corresponding classical geometric structures on representation schemes $\operatorname{Rep}_N(A)$ parametrizing \(N\)-dimensional representations for all $N\geq 1$. In particular, a suitable notion of ``noncommutative Poisson structure'' on $A$ should give rise to ordinary Poisson brackets on the coordinate rings $\mathcal{O}(\operatorname{Rep}_N(A))$. Three such structures have been studied in the literature: the double Poisson brackets of Van den Bergh~\cite{van2008double}, the right double Poisson brackets of Fairon and McCulloch~\cite{fairon2023around}, and the wheeled, or \textit{di-twisted}, Poisson brackets of Ginzburg and Schedler~\cite{ginzburg2010differentialoperatorsbvstructures}. The main goal of this paper is to clarify the relationship between these three notions.

Van den Bergh \cite{van2008double} introduced \textit{double Poisson brackets} as linear maps $\{\!\!\{-,-\}\!\!\}\colon A\otimes A\longrightarrow A\otimes A$ satisfying noncommutative analogs of skew-symmetry, the Leibniz rule, and the Jacobi identity. Each such bracket induces a $\operatorname{GL}_N$-invariant Poisson bracket on the coordinate ring $\mathcal{O}(\operatorname{Rep}_N(A))$, hence satisfies the Kontsevich--Rosenberg principle.

Double Poisson brackets have since become a central tool in noncommutative Poisson geometry. They and their quasi-Poisson analogs have revealed important connections with, among other things, surface groups and character varieties (\cite{massuyeau2014quasi}, \cite{turaev2014poisson}, \cite{alekseev2018goldman}, \cite{gekhtman2024double}, \cite{fairon2025wild}); integrable systems and Hamiltonian PDEs (\cite{ovenhouse2020noncommutative}, \cite{desole2015double}, \cite{alvarez2023noncommutative}, \cite{fairon2022morphisms}, \cite{arthamonov2017modified}); the associative and classical Yang--Baxter equations (\cite{schedler2009poisson}, \cite{odesskii2013double}, \cite{odesskii2014parameter}); Rota--Baxter operators (\cite{goncharov2022double}, \cite{fairon2024modified}); and pre-Calabi--Yau algebras (\cite{iyudu2021precalabiYau}, \cite{fernandez2022qpcy}). The internal structure theory has also been developed in several directions: classifications on commutative algebras (\cite{powell2016double}) and on finite-dimensional algebras (\cite{vandeweyer2008double}, \cite{sharygin2025low}); cohomology (\cite{pichereau2008double}, \cite{fairon2025cohomology}); deformation quantization and formality (\cite{safonkin2025doublestarproduct}); as well as homotopical and derived versions (\cite{fernandez2021cyclic}, \cite{quesney2023balanced}, \cite{leray2025pre}, \cite{liu2025shifted}, \cite{pridham2020shifted}).

Ginzburg and Schedler \cite{ginzburg2010differentialoperatorsbvstructures} introduced the more general framework of \textit{di-twisted Poisson brackets} (called \textit{wheeled Poisson brackets} in \emph{loc.\ cit.}) on a certain graded algebra $\O(A)$, which we call the double coordinate ring of $A$ following \cite{safonkin2025doublestarproduct}. The algebra $\O(A)$ is built out of tensor powers of $A$, the symmetric algebra $\operatorname{S}(A/[A,A])$, and the group algebras of the symmetric groups; it naturally parametrizes $\operatorname{GL}_N$-equivariant operations on the representation spaces via a generalization of the Le Bruyn--Procesi theorem established in \cite{safonkin2025doublestarproduct}. Di-twisted Poisson brackets on $\O(A)$ induce $\operatorname{GL}_N$-invariant Poisson structures on each $\operatorname{Rep}_N(A)$ and are seemingly the most general structures with this property.

Later, Fairon and McCulloch \cite{fairon2023around} introduced \textit{right double Poisson brackets}, which use a different bimodule structure on $A\otimes A$ compared to the double Poisson brackets of Van den Bergh and produce $\operatorname{GL}_N$-invariant Poisson brackets on representation spaces via a different formula. The connection between double Poisson brackets and di-twisted Poisson brackets was mentioned in \cite{ginzburg2010differentialoperatorsbvstructures} and proved in \cite{fernández2025symplecticwheelgebrasnoncommutativegeometry}. The intuition that the wheeled Poisson brackets of Ginzburg and Schedler are the most general structures satisfying the Kontsevich--Rosenberg principle for Poisson brackets suggests that there must also be a connection between them and the right double Poisson brackets of Fairon and McCulloch.

Before we describe this connection, let us point out why the original double and right double brackets are not enough to capture all $\operatorname{GL}_N$-invariant Poisson structures on representation schemes. For example, for $A=\kk[x]$, the representation scheme $\operatorname{Rep}_N(A)$ is canonically identified with $\mathfrak{gl}_N^*$, so $\mathcal{O}(\operatorname{Rep}_N(A))\simeq \operatorname{S}(\mathfrak{gl}_N)$, and the Kirillov--Kostant--Souriau bracket on $\operatorname{S}(\mathfrak{gl}_N)$ comes from the linear double Poisson bracket $\{\!\!\{x,x\}\!\!\}=1\otimes x-x\otimes 1$ on $A$. One can modify the KKS bracket by multiplying it by a Casimir, say $E_{11}+\ldots+E_{NN}$, and obtain another $\operatorname{GL}_N$-invariant Poisson bracket given by
\begin{equation*}
    \{E_{ij},E_{kl}\}_N=(\delta_{kj}E_{il}-\delta_{il}E_{kj})\sum\limits_{p=1}^N E_{pp}.
\end{equation*}
On the other hand, this new Poisson bracket still depends on $N$ ``tamely'', e.g., $N$ does not enter the formula explicitly but only through the summation limit, and therefore the Poisson bracket should come from a ``universal object'' that does not depend on $N$. It is indeed induced by a di-twisted Poisson bracket on $\O(A)$, which does not depend on $N$, yet it is not induced by any double or right double bracket on $A$: the trace $\sum_p E_{pp}$ has no place in the target $A\otimes A$ of double/right double brackets but fits naturally inside the symmetric algebra $\operatorname{S}(A/[A,A])$. This mechanism is responsible for the appearance of the tensor factor $\operatorname{S}(A/[A,A])$ in the generalized versions of double and right double brackets below.

The present paper describes the connection between the di-twisted Poisson brackets of Ginzburg and Schedler and Van den Bergh's double and Fairon--McCulloch's right double Poisson brackets through a notion we call \textit{coupled double Poisson brackets}: a pair consisting of a generalization of a Van den Bergh double bracket and a generalization of a Fairon--McCulloch right double bracket, each valued in $A\otimes A\otimes \operatorname{S}(A/[A,A])$ rather than in $A\otimes A$, and subject to a cross-Jacobi compatibility that couples the two (Definition~\ref{def16}). The main result reads as follows.

\begin{theorem}[{Theorem~\ref{th1} below}]
Let $A$ be a finitely generated associative $\kk$-algebra. Restriction to $A\otimes A\subset \O(A)_1\otimes \O(A)_1$ yields a bijection between di-twisted Poisson brackets on $\O(A)$ and pairs $\bigl(\{\!\!\{-,-\}\!\!\}_{\operatorname{id}},\,\{\!\!\{-,-\}\!\!\}_{(12)}\bigr)$ of coupled double Poisson brackets on $A$. The correspondence is given by
\begin{equation*}
    \{a,b\}=\{\!\!\{a,b\}\!\!\}_{\operatorname{id}}\otimes \operatorname{id}_2+\{\!\!\{a,b\}\!\!\}_{(12)}\otimes (12)\in\O(A)_2,\qquad a,b\in A,
\end{equation*}
where $\{\!\!\{-,-\}\!\!\}_{\operatorname{id}}\colon A\otimes A\longrightarrow A\otimes A\otimes \operatorname{S}(A/[A,A])$ is a right double bracket and $\{\!\!\{-,-\}\!\!\}_{(12)}\colon A\otimes A\longrightarrow A\otimes A\otimes \operatorname{S}(A/[A,A])$ is a double bracket.
\end{theorem}

In one direction, Theorem~\ref{th1} is essentially a matter of unpacking the axioms: a di-twisted Poisson bracket restricted to $A\otimes A$ automatically takes values in $\O(A)_2=A\otimes A\otimes \operatorname{S}(A/[A,A])\otimes \kk[S(2)]$, and the decomposition $\kk[S(2)]=\kk\cdot \operatorname{id}_2\oplus \kk\cdot (12)$ splits it into two components, while the di-twisted skew-symmetry, Leibniz rule, and Jacobi identity translate into the corresponding conditions, together with the cross-Jacobi compatibility, for the resulting pair. The converse direction is more delicate: given a coupled pair on $A$, one must extend the two brackets from $A\otimes A$ to all of $\O(A)$ by forced prescriptions coming from bimodule equivariance, the graded Leibniz rule, and compatibility with the map $\pi\colon \O(A)\to \O(A)$ that encodes the multiplication in $A$, and then verify that the resulting bracket satisfies the di-twisted Jacobi identity throughout. This is carried out in Section~\ref{sec_proof}.

Theorem~\ref{th1} has two immediate consequences. Combined with Proposition~\ref{prop3}, it yields the Kontsevich--Rosenberg principle for coupled double Poisson brackets: every coupled pair on $A$ produces a $\operatorname{GL}_N$-invariant Poisson bracket on each coordinate ring $\mathcal{O}(\operatorname{Rep}_N(A))$, and the formulas of Propositions~\ref{prop1} and~\ref{prop2} are recovered as the two extreme cases singled out in Remark~\ref{rem2} (Corollary~\ref{cor_KR}). And, as in the double and right double settings, verifying the coupled double Poisson axioms reduces to checking them on a set of algebra generators of $A$, which is a finite condition when $A$ is finitely generated (Corollary~\ref{cor1}).

The paper also investigates two important classes of coupled double Poisson brackets, both in the setting where the two brackets land in $A\otimes A$ (the third case of Remark~\ref{rem2}). In the linear case (Section~\ref{sec_linear}), we show that linear coupled double Poisson brackets on the tensor algebra $\T(V)$ of a finite-dimensional vector space $V$ are classified by a new algebraic structure on $V$, which we call a \textit{Poisson-left-pre-Lie algebra}: the data of an associative product $\cdot$ and a left pre-Lie bracket $\{-,-\}$ on $V$ satisfying the cyclic identity and the Leibniz rule
\begin{equation*}
\{ab,c\}=\{ba,c\},\qquad \{a,bc\}=\{a,b\}c+b\{a,c\}.
\end{equation*}
In the quadratic case (Section~\ref{sec1}), we observe that any pair of skew-symmetric linear maps $\mathcal{R},\,r\colon V\otimes V\longrightarrow V\otimes V$ such that $\mathcal{R}$ satisfies the associative Yang--Baxter equation (AYBE), $r$ satisfies the classical Yang--Baxter equation (CYBE), and the compatibility condition
\begin{equation*}
[\mathcal{R}^{12},\,r^{13}+r^{23}]=0
\end{equation*}
holds, gives rise to a pair of quadratic coupled double Poisson brackets on $\T(V)$. Building on Sokolov's classification of skew-symmetric solutions of the AYBE in dimension three \cite{sokolov2013classification}, we exhibit explicit families of such compatible pairs for $\dim V\leq 3$ and conjecture that, for $\dim V=3$, our list---together with the transposed duals of its entries and their scalar multiples---exhausts all compatible pairs modulo the conjugation action of $\operatorname{GL}_3(\kk)$. We also construct a multiparameter family of quadratic coupled double Poisson brackets on $\kk\langle x_1,\ldots,x_n\rangle$ for arbitrary $n$.

We close with a remark on the Kontsevich--Rosenberg principle. The principle admits at least two variations: a derived version, in which representation schemes are replaced by derived representation schemes \cite{berest2013derived}, \cite{berest2012noncommutative}, \cite{berest2014derived}; and a B-C-D-type variant \cite{olshanski2023double}, in which representation schemes are replaced by certain subschemes carrying a canonical action of an orthogonal or symplectic group, i.e., a classical group of type B, C, or D. The present paper focuses on the original version of the principle, but we expect that the results and methods can be adapted to these variants as well. Other natural directions that lie beyond the present work are quasi-Poisson analogs of coupled double Poisson brackets following \cite{van2008non} and coupled versions of double Poisson vertex algebras following \cite{desole2015double}. We hope to return to these problems in future work.

\medskip

\textbf{Outline of the paper.} Section~\ref{section_preliminaries} recalls the definitions of double and right double Poisson brackets, the representation scheme $\rep$, the double coordinate ring $\O(A)$, and Ginzburg--Schedler's di-twisted Poisson brackets. Section~\ref{sec_coupled} introduces the notion of coupled double Poisson brackets (Definition~\ref{def16}), states the main bijection with di-twisted Poisson brackets on $\O(A)$ (Theorem~\ref{th1}), and records the resulting Kontsevich--Rosenberg principle (Corollary~\ref{cor_KR}) together with the reduction of the identities to a set of generators (Corollary~\ref{cor1}). Section~\ref{sec_proof} contains the proofs of Theorem~\ref{th1} and Corollary~\ref{cor1}. Section~\ref{sec_linear} treats the linear case on the tensor algebra $\T(V)$ and identifies it with Poisson-left-pre-Lie algebra structures (Theorem~\ref{thm_linear}). Section~\ref{sec1} is devoted to the quadratic case: the Yang--Baxter compatibility criterion (Proposition~\ref{prop5}), the conjectural classification on $\kk^3$ following Sokolov's list, and a multiparameter family on $\kk^n$. 

\medskip

\textbf{Acknowledgments.} I am grateful to Maxime Fairon, Vladimir Roubtsov, and Michael Semenov-Tian-Shansky for many valuable comments and stimulating discussions. The author was partially supported by the European Research Council (ERC), Grant Agreement No.~101041499.

\addtocontents{toc}{\protect\setcounter{tocdepth}{1}}

\section{Preliminaries and notation}\label{section_preliminaries}

Throughout the paper we assume that $\kk$ is an algebraically closed field of characteristic zero. We denote the tensor product of vector spaces over $\kk$ by $\otimes$. All algebras in the paper are assumed to be unital unless otherwise stated.

By $A$ we will always mean a finitely generated algebra over $\kk$, by $A_{\natural}$ we denote the quotient vector space $\faktor{A}{[A,A]}$. We will denote the map $A\longrightarrow A_{\natural}$ by $a\mapsto \overline{a}$. Below we work with tensor powers of the algebra $A$ and use the following notation.

 An element $a=\sum\limits_i a_i'\otimes a_i''\in A^{\otimes 2}$ is usually written as $a=a'\otimes a''$. We will call these $a'$ and $a''$ \textit{the Sweedler's components} of $a$. We use similar notation for the tensor cube or any other tensor power of $A$, e.g., $a'\otimes a''\otimes a'''=\sum\limits_i a_i'\otimes a_i''\otimes a_i'''\in A^{\otimes 3}$.

For any $n$ we consider the natural action of the symmetric group on $n$ letters $S(n)$ on $A^{\otimes n}$ given by 
\begin{equation}
    \sigma (a_1\otimes\ldots\otimes a_n):=a_{\sigma^{-1}(1)}\otimes\ldots\otimes a_{\sigma^{-1}(n)},\ \ \ \sigma\in S(n).
\end{equation}

\subsection{Double Poisson brackets}
Recall the definition of outer and inner bimodule structures on $A\otimes A$
\begin{gather}
    a(x'\otimes x'')b:=(ax')\otimes (x''b),\\
    a*(x'\otimes x'')*b:=(x'b)\otimes (ax''),
\end{gather}
where $a,b,x',x''\in A$.

Note that the permutation $(12)\in S(2)$ interchanges these two bimodule structures
\begin{equation}
    (12)(axb)=a*(12)x*b,\ \ \ \ \ \ \ a,b\in A,\ \ \ \ \ x\in A^{\otimes 2}.
\end{equation}

\begin{definition}[Van den Bergh, \cite{van2008double}]
A \emph{double Poisson bracket} on $A$ is a linear map 
$$
\{\!\!\{-,-\}\!\!\}: A\otimes A \longrightarrow  A\otimes A
$$
satisfying the following three conditions. 

1. \emph{Skew-symmetry}: for $a,b\in A$,
\begin{equation}
\{\!\!\{a,b\}\!\!\}=-(12)\{\!\!\{b,a\}\!\!\}.
\end{equation}

2. \emph{Outer Leibniz rule}: for $a,b,c\in A$, 
\begin{equation}
\{\!\!\{a, bc\}\!\!\}=\{\!\!\{a,b\}\!\!\} c+ b\{\!\!\{a,c\}\!\!\},
\end{equation}
i.e., $\{\!\!\{a,-\}\!\!\}$ is a derivation $A\longrightarrow A\otimes A$ with respect to the outer bimodule structure; then $\{\!\!\{-,b\}\!\!\}$ is necessarily a derivation $A\longrightarrow A\otimes A$ with respect to the inner bimodule structure.
\vspace{5pt}

3. \emph{Double Jacobi identity}: $\Jac(a,b,c)=0$ for $a,b,c\in A$, where
\begin{equation}
\Jac(a,b,c):=\Bigl\{\!\!\!\Bigl\{ a,\{\!\!\{b,c\}\!\!\}\Bigr\}\!\!\!\Bigr\}_{\mathbf L} +(123) \Bigl\{\!\!\!\Bigl\{ b,\{\!\!\{ c,a\}\!\!\}\Bigr\}\!\!\!\Bigr\}_{\mathbf L} +(123)^2 \Bigl\{\!\!\!\Bigl\{ c,\{\!\!\{ a,b\}\!\!\}\Bigr\}\!\!\!\Bigr\}_{\mathbf L},
\end{equation}
and $(123)$ is the cyclic permutation $1\to 2\to 3\to 1$, which pushes everything to the right
$$
(123)\cdot (x'\otimes x''\otimes x''')=x'''\otimes x'\otimes x'',
$$
and for $x=x'\otimes x''\in A^{\otimes 2}$,
$$
\{\!\!\{ a,x\}\!\!\}_{\mathbf L}:= \{\!\!\{ a,x'\}\!\!\}\otimes x''\in A^{\otimes 3}.
$$
\end{definition}

\subsection{Right-double Poisson brackets}

Here we use different bimodule structures on $A\otimes A$, namely the right and left bimodule structures
\begin{gather}
    a\cdot_r(x'\otimes x'')\cdot _rb:=x'\otimes (ax''b),\\
    a*_l(x'\otimes x'')*_lb:=(ax'b)\otimes x'',
\end{gather}
where $a,b,x',x''\in A$.

The permutation $(12)\in S(2)$ interchanges them
\begin{equation}
    (12)(a\cdot_rx\cdot_rb)=a*_l(12)x*_lb,\ \ \ \ \ \ \ a,b\in A,\ \ \ \ \ x\in A^{\otimes 2}.
\end{equation}

\begin{definition}[Fairon-McCulloch, \cite{fairon2023around}]
    A \emph{right double Poisson bracket} on $A$ is a linear map 
$$
\{\!\!\{-,-\}\!\!\}: A\otimes A \longrightarrow  A\otimes A
$$
satisfying the following three conditions. 

1. \emph{Skew-symmetry}: for $a,b\in A$,
\begin{equation}
\{\!\!\{a,b\}\!\!\}=-(12)\{\!\!\{b,a\}\!\!\}.
\end{equation}

2. \emph{Right Leibniz rule}: for $a,b,c\in A$, 
\begin{equation}
\{\!\!\{a, bc\}\!\!\}=\{\!\!\{a,b\}\!\!\} c+ b\{\!\!\{a,c\}\!\!\},
\end{equation}
i.e., $\{\!\!\{a,-\}\!\!\}$ is a derivation $A\longrightarrow A\otimes A$ with respect to the right bimodule structure; then $\{\!\!\{-,b\}\!\!\}$ is necessarily a derivation $A\longrightarrow A\otimes A$ with respect to the left bimodule structure.
\vspace{5pt}

3. \emph{Right double Jacobi identity\footnote{The authors used the term $(12)$-weak double Jacobiator for the left-hand side of the identity.}}: $\Jac(a,b,c)-(12)\Jac(b,a,c)=0$ for $a,b,c\in A$. 
\end{definition}

\subsection{Representation spaces \texorpdfstring{$\rep$}{RepNA}}

Let $N$ be a positive integer. 

\begin{definition}
    Representation space $\rep$ is by definition the affine $\kk$-scheme representing the functor 
    \begin{align}
        &CAlg\longrightarrow Set,\\
        &R\mapsto \operatorname{Hom}_{\operatorname{Alg}}\left(A,\operatorname{Mat}_N(R)\right),
    \end{align}
    where $CAlg$ is the category of commutative unital algebras.
\end{definition}

The coordinate ring $\orep$ of $\rep$ is the commutative algebra generated by the symbols $a_{ij}$ for $a\in A$ and $i,j\in\{1,\dots,N\}$, which are $\kk$-linear in $a$  and  are subject to the relations
\begin{equation}\label{f17}
(ab)_{ij}=\sum_{k=1}^N a_{ik}b_{kj}, \qquad a,b\in A, \quad i,j=1,\dots,N
\end{equation}
together with
\begin{equation}\label{f36}
    (1)_{ij}=\delta_{ij} \qquad i,j=1,\dots,N,
\end{equation}
where $1$ is the identity in $A$. 

\begin{proposition}[\cite{van2008double}, Proposition 1.2]\label{prop1}
Let $A$ be an associative $\kk$-algebra and $\{\!\!\{-,-\}\!\!\}$ be a double Poisson bracket on $A$. For each $N=1,2,\dots$, the formula
\begin{equation}
\{a_{ij},b_{kl}\}:=\{\!\!\{ a,b\}\!\!\}'_{kj}\{\!\!\{ a,b\}\!\!\}''_{il}, \qquad a,b\in A, \quad i,j,k,l\in\{1,\dots,N\},
\end{equation}
gives rise to a Poisson bracket on the commutative algebra $\orep$. 
\end{proposition}

\begin{proposition}[\cite{fairon2023around}, Theorem 4.11]\label{prop2}
Let $A$ be an associative $\kk$-algebra and $\{\!\!\{-,-\}\!\!\}$ be a right double Poisson bracket on $A$. For each $N=1,2,\dots$, the formula
\begin{equation}
\{a_{ij},b_{kl}\}:=\{\!\!\{ a,b\}\!\!\}'_{ij}\{\!\!\{ a,b\}\!\!\}''_{kl}, \qquad a,b\in A, \quad i,j,k,l\in\{1,\dots,N\},
\end{equation}
gives rise to a Poisson bracket on the commutative algebra $\orep$. 
\end{proposition}

\subsection{\texorpdfstring{$\S$}{S}-bimodules and the double coordinate ring}

As one can readily verify, the Poisson brackets from Propositions \ref{prop1} and \ref{prop2} above are $\GL$-invariant. It is, in fact, possible to describe all $\GL$-invariant linear maps $\orep^{\otimes n}\longrightarrow \orep$ for any $n$. This description naturally leads to an algebra $\O(A)$, which we call the double coordinate ring, see \cite{safonkin2025doublestarproduct} for details. The same algebra was introduced earlier by Ginzburg and Schedler \cite{ginzburg2010differentialoperatorsbvstructures}, where it is denoted by $\mathcal{F}(A)$ and viewed as a ``Fock space'' attached to $A$.

\begin{definition}
    Let $A$ be an associative algebra over $\kk$. We set
    \begin{align}
        \hspace{30pt} \O(A)=\bigoplus\limits_{n\geq 0}\O(A)_n,\hspace{25pt} \O(A)_n:=A^{\otimes n}\otimes \s\otimes \kk[S(n)].
    \end{align}
    The multiplication on $\O(A)$ is given by 
    \begin{align}
        (a \otimes f \otimes u)(b \otimes g \otimes v):=(a\otimes b) \otimes f\cdot g \otimes u\times v
    \end{align}
    for $a\in A^{\otimes n}$, $b\in A^{\otimes m}$, $f,g\in\s$, $u\in S(n)$, $v\in S(m)$, where $u\times v\in S(n+m)$ is the permutation that permutes the first $n$ nodes according to $u$ and the last $m$ nodes according to $v$.
\end{definition}

Each $\O(A)_n$ is an $S(n)$-bimodule via
\begin{align}
    u\cdot (a_1\otimes \ldots\otimes a_n\otimes f\otimes \tau)\cdot v:=a_{u^{-1}(1)}\otimes \ldots\otimes a_{u^{-1}(n)}\otimes f \otimes u \tau v
\end{align}
and the multiplication $\O(A)_n\otimes \O(A)_m\longrightarrow \O(A)_{n+m}$ is a map of $S(n)\times S(m)\subset S(n+m)$ bimodules. We will denote the diagonal $S(n)$-action on $\O(A)_n$ by $\operatorname{Ad}$, i.e., $\operatorname{Ad}(u)(\a):=u\cdot \a\cdot u^{-1}$ for $\a\in\O(A)_n$.

The algebra structure on $\O(A)$ does not involve the multiplication in $A$, so we need an additional structure to keep track of it: a graded linear map $\pi:\O(A)\longrightarrow \O(A)$ of degree $-1$ given by
\begin{equation}
    \pi(a\otimes f\otimes u)=\begin{cases}
        (a_2\otimes\ldots\otimes a_n)\otimes \overline{a_1}\cdot f\otimes u_*, &\text{if}\ u(1)=1,\\[5pt]
        m_{1,u(1)}(a) \otimes f\otimes u_*, &\text{if}\ u(1)>1,
    \end{cases}
\end{equation}

where 
\begin{itemize}
    \item $a=a_1\otimes\ldots\otimes a_n\in A^{\otimes n}$, $f\in\s$, and $u\in S(n)$;
    
    \item the map $m_{1,k}:A^{\otimes n}\rightarrow A^{\otimes n-1}$ multiplies the $k$-th component by the first one, i.e., $$m_{1,k}(a_1\otimes\ldots\otimes a_n)=a_2\otimes\ldots\otimes a_{k-1}\otimes a_1a_k\otimes a_{k+1}\otimes\ldots\otimes a_n.$$

    \item $u_*\in S(n-1)$ is defined as follows. Consider $u$ as a collection of $n$ nodes in the first row, joined by $n$ edges to $n$ nodes in the second row. In the case $u=\operatorname{id}_1\times w\in S(n)$ we set by definition $u_*=w\in S(n-1)$. If $u(1)>1$, we add one extra edge connecting the first vertices in the upper and lower rows. Then $u_*\in S(n-1)$ is the permutation defined by the paths in this new graph.
\end{itemize}

For example, $\pi(a\otimes f\otimes \operatorname{id}_1)=\mathds{1}\otimes\overline{a}\cdot f\otimes\operatorname{id}_0$, $\pi(a\otimes b\otimes f\otimes (12))=ab\otimes f\otimes\operatorname{id}_1$, and $\pi(\mathds{1}\otimes f\otimes\operatorname{id}_0)=0$ for $a,b\in A$, $f\in\s$.

The structure on $\O(A)$ responsible for the unit in $A$ is the operator of left multiplication by $1\otimes \mathds{1}\otimes \id_1$, where $1\in A$ and $\mathds{1}\in\s$ are the units. However, we will not use it below. 

Let us mention a link between the double coordinate ring and representation spaces. Specifically, we explain how to produce elements of $\orep$ out of homogeneous elements of $\O(A)$ (and a tuple of matrices). Assume that two tuples of indices ranging from $1$ to $N$ are given, say $\mathbf{i}=(i_1,\ldots,i_n)$ and $\mathbf{j}=(j_1,\ldots,j_n)$. Take any $a_1,\ldots, a_n\in A$, a permutation $u\in S(n)$, and any $f_1,\ldots, f_m\in A_{\natural}$. Then the element $\a=(a_1\otimes\ldots\otimes a_n) \otimes f_1\cdot\ldots\cdot f_m\otimes u$ belongs to $\O(A)_n$ and the element of the coordinate ring $\orep$ corresponding to $\mathbf{i}$, $\mathbf{j}$, and $\a$, denoted by $\a_{\mathbf{i}\mathbf{j}}$, is by definition
\begin{equation}\label{f28}
    \a_{\mathbf{i}\mathbf{j}}:=(a_1)_{i_{u^{-1}(1)}j_1}\ldots (a_n)_{i_{u^{-1}(n)}j_n}\tr(f_1)\ldots\tr(f_m)\in\orep,
\end{equation}
where $\tr(f)=\sum\limits_{i=1}^Nf_{ii}$.

\subsection{Di-twisted Poisson brackets on the double coordinate ring}

Below, for any permutation $\tau\in S(n)$ and any integers $k_1,\ldots,k_n\in\mathbb{Z}_{\geq 0}$, we denote by $\tau^{k_1,\ldots,k_n}\in S(k_1+\ldots+k_n)$ the permutation that splits all $k_1+\ldots+k_n$ nodes in $n$ blocks of lengths $k_1,\ldots,k_n$, and permutes these blocks according to $\tau$.

\begin{definition}[Ginzburg-Schedler, \cite{ginzburg2010differentialoperatorsbvstructures}]\label{def2}
    A \textit{di-twisted Poisson bracket}\footnote{The authors called it \textit{wheeled Poisson bracket}.} on $\O(A)$ is a homogeneous graded linear map $\{-,-\}:\O(A)\otimes \O(A)\longrightarrow\O(A)$ of degree zero such that 
    
    \begin{itemize}

    \item $\{-,-\}:\O(A)_n\otimes \O(A)_m\longrightarrow\O(A)_{n+m}$ is a map of $S(n)\times S(m)\subset S(n+m)$ bimodules
    \vspace{7pt}

    \item for any homogeneous $\a,\b \in\O(A)$
    \begin{align}
        \{\pi(\a),\b\}=\pi\{\a,\b\}
    \end{align}
    
    \item \textit{Skew-symmetry:} for any homogeneous $\a,\b\in\O(A)$
    \begin{equation}
    \{\b,\a\}=-\operatorname{Ad}((12)^{|\a|,|\b|})\{\a,\b\}.    
    \end{equation}

    \item \textit{Leibniz rule:} for any homogeneous $\a,\b,\cc\in\O(A)$
    \begin{align}
        \{\a,\b\cc\}&=\{\a,\b\}\cc+\operatorname{Ad}((12)^{|\b|,|\a|,|\cc|})\Bigl(\b\{\a,\cc\}\Bigr),\\[7pt]
        \{\a\b,\cc\}&=\operatorname{Ad}((23)^{|\a|,|\cc|,|\b|})\Bigl(\{\a,\cc\}\b\Bigr)+\a\{\b,\cc\}
    \end{align}

    \item \textit{Jacobi identity:} for any homogeneous $\a,\b,\cc\in\O(A)$
    \begin{equation}
        \bigl\{ \a,\{\b,\cc\}\bigr\}+\operatorname{Ad}((123)^{|\b|,|\cc|,|\a|}) \bigl\{ \b,\{ \cc,\a\}\bigr\} +\operatorname{Ad}((132)^{|\cc|,|\a|,|\b|}) \bigl\{ \cc,\{ \a,\b\}\bigr\}=0.
    \end{equation}
    \end{itemize}

\end{definition}

Note that any di-twisted Poisson bracket on $\O(A)$ is uniquely determined by its values on $A\otimes A\subset \O(A)_1\otimes \O(A)_1$.

\begin{proposition}[\cite{ginzburg2010differentialoperatorsbvstructures}]\label{prop3}
    Let $A$ be an associative algebra and let $\{-,-\}$ be a di-twisted Poisson bracket on $\O(A)$. For each $N=1,2,\dots$, the formula
\begin{equation}
\{\a_{\mathbf{ij}},\b_{\mathbf{kl}}\}_N:=\{\a,\b\}_{\mathbf{i}\sqcup\mathbf{k},\mathbf{j}\sqcup\mathbf{l}},
\end{equation}
where $\sqcup$ stands for concatenation of tuples, gives rise to a $\GL$-invariant Poisson bracket $\{-,-\}_N$ on the commutative algebra $\orep$.
\end{proposition}

In a certain asymptotical sense, any $\GL$-invariant Poisson bracket on $\orep$ can be obtained in such way, see Proposition 6.1 and Theorem 5.1 in \cite{safonkin2025doublestarproduct}. Roughly speaking, a $\GL$-invariant Poisson bracket on $\orep$ comes from a di-twisted Poisson bracket on $\O(A)$, if the former depends on $N$ in a controllable way, e.g., polynomially. 

Now it is natural to ask ``How to link Proposition \ref{prop3} to Proposition \ref{prop1} and Proposition \ref{prop2}?'' The connection between di-twisted Poisson brackets on $\O(A)$ and double Poisson brackets on $A$ was mentioned without proof in \cite{ginzburg2010differentialoperatorsbvstructures}, see Remark 3.5.19. Namely, any double Poisson bracket on $A$ can be canonically extended to a di-twisted Poisson bracket on $\O(A)$. This was proved in a recent paper \cite{fernández2025symplecticwheelgebrasnoncommutativegeometry}, see Proposition 6.8 there. Below, we establish the remaining connection which allows us to unite double Poisson brackets and right double Poisson brackets. 

 \addtocontents{toc}{\protect\setcounter{tocdepth}{2}}

\section{Coupled double Poisson brackets}\label{sec_coupled}

A pair of coupled double Poisson brackets is, loosely speaking, a pair consisting of a double Poisson bracket in the sense of Van den Bergh and a right double bracket in the sense of Fairon-McCulloch, subject to certain compatibility constraints that take the form of the vanishing of certain double cross-Jacobiators. However, we will extend both of the definitions to the case when there is a third tensor factor $\s$ in the target of the brackets. The necessity of this is clear from the following example.

\begin{example}
    Let $A=\kk[x]$ be the algebra of polynomials in one variable. Then $\rep$ can be conveniently identified with $\mathfrak{gl}_N^*$, so we may identify $\orep$ with $\operatorname{S}(\mathfrak{gl}_N)$. Recall that the commutator bracket on $\operatorname{S}(\mathfrak{gl}_N)$ comes from the linear double Poisson bracket on $A$ defined by $\{\!\!\{x,x\}\!\!\}=1\otimes x-x\otimes 1$. Next, consider the Poisson bracket given on generators by 
    \begin{align}
        \{E_{ij},E_{kl}\}_N:=[E_{ij},E_{kl}] \sum\limits_{p=1}^N E_{pp}=(\delta_{kj}E_{il}-\delta_{il}E_{kj})\sum\limits_{p=1}^N E_{pp}.
    \end{align}
    This bracket is essentially independent of $N$\footnote{In \cite{safonkin2025doublestarproduct} it was called \textit{stability in $N$}}, so it must come from a di-twisted Poisson bracket on $\O(A)$. This di-twisted Poisson bracket $\{-,-\}$ is uniquely defined on $A\otimes A\subset \O(A)_1\otimes \O(A)_1$ by
    \begin{align}
        \{x,x\}=(1\otimes x-x\otimes 1)\otimes \overline{x}\otimes (12).
    \end{align}
\end{example}

Below we will consider certain linear maps $\{\!\!\{-,-\}\!\!\}:A\otimes A\longrightarrow A\otimes A\otimes \s$. We will use the following Sweedler's notation
\begin{align}
    \{\!\!\{a,b\}\!\!\}=\{\!\!\{a,b\}\!\!\}'\otimes \{\!\!\{a,b\}\!\!\}''\otimes \{\!\!\{a,b\}\!\!\}''',
\end{align}
where $\{\!\!\{a,b\}\!\!\}',\{\!\!\{a,b\}\!\!\}''\in A$ and $\{\!\!\{a,b\}\!\!\}'''\in \s$.

We will speak about $A$-bimodule structures on $A\otimes A\otimes \s$, apply the flip that permutes the first two components, etc., as if there is no third factor $\s$---this simplifies notation and will not lead to confusion; e.g., $(12)\{\!\!\{a,b\}\!\!\}=\{\!\!\{a,b\}\!\!\}''\otimes \{\!\!\{a,b\}\!\!\}'\otimes \{\!\!\{a,b\}\!\!\}'''$.

\begin{definition}\label{def15}
    Let $\{\!\!\{-,-\}\!\!\}:A\otimes A\longrightarrow A\otimes A\otimes \s$ be a cyclically skew-symmetric linear map, i.e., a linear map such that $\{\!\!\{a,b\}\!\!\}=-(12)\{\!\!\{b,a\}\!\!\}$ for any $a,b\in A$. We will say that $\{\!\!\{-,-\}\!\!\}$ is a \textit{double bracket} if it is a derivation in its second argument with respect to the outer $A$-bimodule structure on $A\otimes A\otimes \s$, i.e.,
    \begin{align}
        \{\!\!\{a,bc\}\!\!\}=\{\!\!\{a,b\}\!\!\}'\otimes \{\!\!\{a,b\}\!\!\}''c\otimes \{\!\!\{a,b\}\!\!\}'''+b\{\!\!\{a,c\}\!\!\}'\otimes \{\!\!\{a,c\}\!\!\}''\otimes \{\!\!\{a,c\}\!\!\}''',
    \end{align}
    and a \textit{right double bracket} if it is a derivation in its second argument with respect to the right $A$-bimodule structure on $A\otimes A\otimes \s$, i.e.,
    \begin{align}
        \{\!\!\{a,bc\}\!\!\}=\{\!\!\{a,b\}\!\!\}'\otimes \{\!\!\{a,b\}\!\!\}''c\otimes \{\!\!\{a,b\}\!\!\}'''+\{\!\!\{a,c\}\!\!\}\otimes b\{\!\!\{a,c\}\!\!\}''\otimes \{\!\!\{a,c\}\!\!\}'''.
    \end{align}
\end{definition}

One can easily check that a double bracket $\{\!\!\{-,-\}\!\!\}$ gives rise to a bi-derivation on $\orep$ by the usual formula from \cite{van2008double} 
\begin{align}
(a_{ij},b_{kl})\mapsto \{\!\!\{a,b\}\!\!\}_{kj,il}:=\{\!\!\{a,b\}\!\!\}'_{kj}\{\!\!\{a,b\}\!\!\}''_{il}\tr\Big(\{\!\!\{a,b\}\!\!\}'''\Big)    
\end{align}
and a right double bracket $\{\!\!\{-,-\}\!\!\}$ gives rise to a similar bi-derivation by the formula from Section 4.3 in \cite{fairon2023around} 
\begin{align}
(a_{ij},b_{kl})\mapsto \{\!\!\{a,b\}\!\!\}_{ij,kl}:=\{\!\!\{a,b\}\!\!\}'_{ij}\{\!\!\{a,b\}\!\!\}''_{kl}\tr\Big(\{\!\!\{a,b\}\!\!\}'''\Big)    
\end{align}

Assume from now on that we are given a double bracket $\{\!\!\{-,-\}\!\!\}_{(12)}$ and a right double bracket $\{\!\!\{-,-\}\!\!\}_{\id}$. We set 
\begin{align}
    &\{-,-\}_{\id}:A\otimes A_{\natural}\longrightarrow A\otimes \s,\\[7pt]
    &\{a,\overline{f}\}_{\id}=\{\!\!\{a,f\}\!\!\}'_{\id}\otimes \overline{\{\!\!\{a,f\}\!\!\}''_{\id}}\cdot \{\!\!\{a,f\}\!\!\}'''_{\id}
\end{align}
and 
\begin{align}
    &\{-,-\}_{(12)}:A\otimes A_{\natural}\longrightarrow A\otimes \s,\\[7pt]
    &\{a,\overline{f}\}_{(12)}=\{\!\!\{a,f\}\!\!\}''_{(12)}\{\!\!\{a,f\}\!\!\}_{(12)}'\otimes \{\!\!\{a,f\}\!\!\}'''_{(12)},
\end{align}
where $f\in A$ is an arbitrary lift of $\overline{f}\in A_{\natural}$. One can easily check that these maps are well-defined by using the outer and right Leibniz rules.

Next, we extend these maps to $A\otimes \s\longrightarrow A\otimes \s$ by the (commutative) Leibniz rule for $\s$, i.e.,
  \begin{align}
    \{a,\overline{f_1}\cdot\ldots\cdot\overline{f_r}\}_{\id \, \text{or}\, (12)}&=\sum\limits_{k=1}^r(\overline{f_1}\cdot\ldots\cdot\widehat{\overline{f_k}}\cdot\ldots\cdot\overline{f_r})\{a,\overline{f_k}\}_{\id\, \text{or}\, (12)},
\end{align}
where the $\s$-module structure on $A\otimes \s$ ignores the first factor.

On the other hand, one can check that each of the maps $\{-,-\}_{\id}$, $\{-,-\}_{(12)}$ descends to a well-defined map $A_{\natural}\otimes A_{\natural}\longrightarrow \s$. Namely, we set 
\begin{align}\label{f101}
    &\{-,-\}_{\id}:A_{\natural}\otimes A_{\natural}\longrightarrow \s,\\[7pt]
    &\{\overline{f},\overline{g}\}_{\id}=\overline{\{\!\!\{f,g\}\!\!\}'_{\id}}\cdot\overline{\{\!\!\{f,g\}\!\!\}''_{\id}}\cdot \{\!\!\{f,g\}\!\!\}'''_{\id}
\end{align}
and
\begin{align}\label{f102}
    &\{-,-\}_{(12)}:A_{\natural}\otimes A_{\natural}\longrightarrow \s,\\[7pt]
    &\{\overline{f},\overline{g}\}_{(12)}=\overline{\{\!\!\{f,g\}\!\!\}'_{(12)}\{\!\!\{f,g\}\!\!\}_{(12)}''}\cdot \{\!\!\{f,g\}\!\!\}'''_{(12)}.
\end{align}
Note the difference between them; note also that both of them are skew-symmetric.

Now we can extend each of them to a map $\s\otimes\s\longrightarrow\s$ by the commutative Leibniz rule in both arguments.

So, the first argument of the maps $\{-,-\}_{\id}$ and $\{-,-\}_{(12)}$ belongs either to $A$ or to $\s$. This ambiguity causes no confusion, as it will always be clear what case applies.

For $x=\id$ or $(12)$, $y=\id$ or $(12)$, and $a,b,c\in A$ we set
\begin{align}
    &\Big\{\!\!\!\Big\{a,\{\!\!\{b,c\}\!\!\}_x\Big\}\!\!\!\Big\}_y^{\text{left}}=\{\!\!\{a,\{\!\!\{b,c\}\!\!\}'_{x}\}\!\!\}'_{y}\otimes \{\!\!\{a,\{\!\!\{b,c\}\!\!\}'_{x}\}\!\!\}''_{y}\otimes \{\!\!\{b,c\}\!\!\}''_{x}\otimes \{\!\!\{a,\{\!\!\{b,c\}\!\!\}'_{x}\}\!\!\}'''_{y}\cdot\{\!\!\{b,c\}\!\!\}'''_{x}\in A^{\otimes 3}\otimes \s\\[10pt]
    &\Big\{\!\!\!\Big\{a,\{\!\!\{b,c\}\!\!\}_x\Big\}\!\!\!\Big\}_y^{\text{right}}=\{\!\!\{b,c\}\!\!\}'_{x}\otimes \{\!\!\{a,\{\!\!\{b,c\}\!\!\}''_{x}\}\!\!\}'_{y}\otimes \{\!\!\{a,\{\!\!\{b,c\}\!\!\}''_{x}\}\!\!\}''_{y}\otimes \{\!\!\{a,\{\!\!\{b,c\}\!\!\}''_{x}\}\!\!\}'''_{y}\cdot\{\!\!\{b,c\}\!\!\}'''_{x}\in A^{\otimes 3}\otimes \s\\[10pt]
    &\Big\{\!\!\!\Big\{a,\{\!\!\{b,c\}\!\!\}_x\Big\}\!\!\!\Big\}_y^{\text{aux}}=\{\!\!\{b,c\}\!\!\}'_{x}\otimes\{\!\!\{b,c\}\!\!\}''_{x}\otimes \Big\{a,\{\!\!\{b,c\}\!\!\}'''_{x}\Big\}'_{y}\otimes \Big\{a,\{\!\!\{b,c\}\!\!\}'''_{x}\Big\}''_{y}\in A^{\otimes 3}\otimes \s
\end{align}

\begin{definition}\label{def16}
    We say that a right double bracket $\{\!\!\{-,-\}\!\!\}_{\id}:A\otimes A\longrightarrow A\otimes A\otimes \s$ and a double bracket $\{\!\!\{-,-\}\!\!\}_{(12)}:A\otimes A\longrightarrow A\otimes A\otimes \s$ are \textit{coupled double Poisson brackets} if for any $a,b,c\in A$ one has 
    \begin{align}\leqnomode
    1)\hspace{20pt}&\Big\{\!\!\!\Big\{a,\{\!\!\{b,c\}\!\!\}_{\id}\Big\}\!\!\!\Big\}_{\id}^{\text{left}}+(123)\Big\{\!\!\!\Big\{b,\{\!\!\{c,a\}\!\!\}_{\id}\Big\}\!\!\!\Big\}_{\id}^{\text{left}}+(132)\Big\{\!\!\!\Big\{c,\{\!\!\{a,b\}\!\!\}_{\id}\Big\}\!\!\!\Big\}_{\id}^{\text{left}}\\[7pt]
&+(12)\Bigg[\Big\{\!\!\!\Big\{a,\{\!\!\{b,c\}\!\!\}_{\id}\Big\}\!\!\!\Big\}_{\id}^{\text{right}}+(132)\Big\{\!\!\!\Big\{b,\{\!\!\{c,a\}\!\!\}_{\id}\Big\}\!\!\!\Big\}_{\id}^{\text{right}}+(123)\Big\{\!\!\!\Big\{c,\{\!\!\{a,b\}\!\!\}_{\id}\Big\}\!\!\!\Big\}_{\id}^{\text{right}}\Bigg]\\[7pt]
&+(123)\Bigg[\Big\{\!\!\!\Big\{a,\{\!\!\{b,c\}\!\!\}_{\id}\Big\}\!\!\!\Big\}_{\id}^{\text{aux}}+(123)\Big\{\!\!\!\Big\{b,\{\!\!\{c,a\}\!\!\}_{\id}\Big\}\!\!\!\Big\}_{\id}^{\text{aux}}+(132)\Big\{\!\!\!\Big\{c,\{\!\!\{a,b\}\!\!\}_{\id}\Big\}\!\!\!\Big\}_{\id}^{\text{aux}}\Bigg]\\[7pt]
&+(123)\Bigg[\Big\{\!\!\!\Big\{a,\{\!\!\{b,c\}\!\!\}_{\id}\Big\}\!\!\!\Big\}_{(12)}^{\text{aux}}+(123)\Big\{\!\!\!\Big\{b,\{\!\!\{c,a\}\!\!\}_{\id}\Big\}\!\!\!\Big\}_{(12)}^{\text{aux}}+(132)\Big\{\!\!\!\Big\{c,\{\!\!\{a,b\}\!\!\}_{\id}\Big\}\!\!\!\Big\}_{(12)}^{\text{aux}}\Bigg]=0,\\[25pt]
    2)\hspace{20pt}&\Big\{\!\!\!\Big\{a,\{\!\!\{b,c\}\!\!\}_{\id}\Big\}\!\!\!\Big\}_{(12)}^{\text{left}}+(13)\Big\{\!\!\!\Big\{b,\{\!\!\{c,a\}\!\!\}_{\id}\Big\}\!\!\!\Big\}_{(12)}^{\text{right}}+(132)\Big\{\!\!\!\Big\{c,\{\!\!\{a,b\}\!\!\}_{(12)}\Big\}\!\!\!\Big\}_{\id}^{\text{left}}\\[7pt]
&\hspace{90pt}(23)\Big\{\!\!\!\Big\{c,\{\!\!\{a,b\}\!\!\}_{(12)}\Big\}\!\!\!\Big\}_{\id}^{\text{right}}+\Big\{\!\!\!\Big\{c,\{\!\!\{a,b\}\!\!\}_{(12)}\Big\}\!\!\!\Big\}_{\id}^{\text{aux}}+\Big\{\!\!\!\Big\{c,\{\!\!\{a,b\}\!\!\}_{(12)}\Big\}\!\!\!\Big\}_{(12)}^{\text{aux}}=0\\[25pt]
    3)\hspace{20pt}&\Big\{\!\!\!\Big\{a,\{\!\!\{b,c\}\!\!\}_{(12)}\Big\}\!\!\!\Big\}_{(12)}^{\text{left}}+(123)\Big\{\!\!\!\Big\{b,\{\!\!\{c,a\}\!\!\}_{(12)}\Big\}\!\!\!\Big\}_{(12)}^{\text{left}}+(132)\Big\{\!\!\!\Big\{c,\{\!\!\{a,b\}\!\!\}_{(12)}\Big\}\!\!\!\Big\}_{(12)}^{\text{left}}=0\\[25pt]
\end{align}
\end{definition}

\begin{remark}\label{rem2}
Let $\{\!\!\{-,-\}\!\!\}_{\id}$, $\{\!\!\{-,-\}\!\!\}_{(12)}$ be a pair of coupled double Poisson brackets.
    \begin{enumerate}[label=\theenumi),leftmargin=3ex]
        \item If $\{\!\!\{-,-\}\!\!\}_{\id}$ is zero, then the double bracket $\{\!\!\{-,-\}\!\!\}_{(12)}$ satisfies
        \begin{align}
            \Big\{\!\!\!\Big\{a,\{\!\!\{b,c\}\!\!\}_{(12)}\Big\}\!\!\!\Big\}_{(12)}^{\text{left}}+(123)\Big\{\!\!\!\Big\{b,\{\!\!\{c,a\}\!\!\}_{(12)}\Big\}\!\!\!\Big\}_{(12)}^{\text{left}}+(132)\Big\{\!\!\!\Big\{c,\{\!\!\{a,b\}\!\!\}_{(12)}\Big\}\!\!\!\Big\}_{(12)}^{\text{left}}=0
        \end{align}
        and
        \begin{align}
            \Big\{\!\!\!\Big\{a,\{\!\!\{b,c\}\!\!\}_{(12)}\Big\}\!\!\!\Big\}_{(12)}^{\text{aux}}=0
        \end{align}
        
        Particularly, if $\{\!\!\{-,-\}\!\!\}_{(12)}$ lands in $A\otimes A\subset A\otimes A\otimes \s$, then $\{\!\!\{-,-\}\!\!\}_{(12)}$ is a double Poisson bracket in the sense of Van den Bergh.

        \item If $\{\!\!\{-,-\}\!\!\}_{(12)}$ is zero, then the right double bracket $\{\!\!\{-,-\}\!\!\}_{\id}$ satisfies
        \begin{align}
            &\Big\{\!\!\!\Big\{a,\{\!\!\{b,c\}\!\!\}_{\id}\Big\}\!\!\!\Big\}_{\id}^{\text{left}}+(123)\Big\{\!\!\!\Big\{b,\{\!\!\{c,a\}\!\!\}_{\id}\Big\}\!\!\!\Big\}_{\id}^{\text{left}}+(132)\Big\{\!\!\!\Big\{c,\{\!\!\{a,b\}\!\!\}_{\id}\Big\}\!\!\!\Big\}_{\id}^{\text{left}}\\[7pt]
            &\hspace{50pt}+(12)\Bigg[\Big\{\!\!\!\Big\{a,\{\!\!\{b,c\}\!\!\}_{\id}\Big\}\!\!\!\Big\}_{\id}^{\text{right}}+(132)\Big\{\!\!\!\Big\{b,\{\!\!\{c,a\}\!\!\}_{\id}\Big\}\!\!\!\Big\}_{\id}^{\text{right}}+(123)\Big\{\!\!\!\Big\{c,\{\!\!\{a,b\}\!\!\}_{\id}\Big\}\!\!\!\Big\}_{\id}^{\text{right}}\Bigg]\\[7pt]
            &\hspace{50pt}+(123)\Bigg[\Big\{\!\!\!\Big\{a,\{\!\!\{b,c\}\!\!\}_{\id}\Big\}\!\!\!\Big\}_{\id}^{\text{aux}}+(123)\Big\{\!\!\!\Big\{b,\{\!\!\{c,a\}\!\!\}_{\id}\Big\}\!\!\!\Big\}_{\id}^{\text{aux}}+(132)\Big\{\!\!\!\Big\{c,\{\!\!\{a,b\}\!\!\}_{\id}\Big\}\!\!\!\Big\}_{\id}^{\text{aux}}\Bigg]=0
        \end{align}

        Particularly, if $\{\!\!\{-,-\}\!\!\}_{\id}$ lands in $A\otimes A\subset A\otimes A\otimes \s$, then $\{\!\!\{-,-\}\!\!\}_{\id}$ is a right double Poisson bracket in the sense of Fairon-McCulloch.

        \item If both $\{\!\!\{-,-\}\!\!\}_{\id}$ and $\{\!\!\{-,-\}\!\!\}_{(12)}$ land in $A\otimes A\subset A\otimes A\otimes \s$, then $\{\!\!\{-,-\}\!\!\}_{\id}$ is a right double Poisson bracket in the sense of Fairon-McCulloch, $\{\!\!\{-,-\}\!\!\}_{(12)}$ is a double Poisson bracket in the sense of Van den Bergh, and they must satisfy the following compatibility condition
        \begin{align}\label{f2}
            &\Big\{\!\!\!\Big\{a,\{\!\!\{b,c\}\!\!\}_{\id}\Big\}\!\!\!\Big\}_{(12)}^{\text{left}}+(13)\Big\{\!\!\!\Big\{b,\{\!\!\{c,a\}\!\!\}_{\id}\Big\}\!\!\!\Big\}_{(12)}^{\text{right}}+(132)\Big\{\!\!\!\Big\{c,\{\!\!\{a,b\}\!\!\}_{(12)}\Big\}\!\!\!\Big\}_{\id}^{\text{left}}\\
            &\hspace{260pt}+(23)\Big\{\!\!\!\Big\{c,\{\!\!\{a,b\}\!\!\}_{(12)}\Big\}\!\!\!\Big\}_{\id}^{\text{right}}=0
        \end{align}
    \end{enumerate}
\end{remark}

\begin{remark}
    Consider the third case of the previous remark, that is $\br_{\id},\br_{(12)}:A\otimes A\longrightarrow A\otimes A$. If one denotes $\RR=\br_{(12)}$ and $r=\br_{\id}$, then the condition that $\br_{(12)}$ and $\br_{\id}$ form a pair of coupled double Poisson brackets can be rephrased by saying that
    \begin{itemize}
        \item $\RR$ is skew-symmetric $\RR^{12}=-\RR^{21}$ and satisfies the associative Yang-Baxter equation
    \begin{align*}\tag{AYBE}
	   \mathcal R^{12}\mathcal R^{13}-\mathcal R^{23}\mathcal R^{12} +\mathcal R^{13}\mathcal R^{23}=0,
    \end{align*}
where $\mathcal R^{ab}$ denotes $\mathcal R$ acting on the $a$-th and $b$-th factors in $A^{\otimes 3}$, e.g., $\mathcal{R}^{21}:=(12)\mathcal{R}(12)$

\item $r$ is skew-symmetric $r^{12}=-r^{21}$ and satisfies the classical Yang-Baxter equation
\begin{align}\tag{CYBE}
    [r^{12},r^{13}] + [r^{12},r^{23}] + [r^{13},r^{23}] =0
\end{align}

\item they satisfy the following compatibility condition
        \begin{align}\label{f33}
            [\RR^{12},r^{13}+r^{23}]=0
        \end{align}
    \end{itemize}
    plus obvious identities coming from the Leibniz rules.

    The equivalence between the associative Yang-Baxter equation and the double Jacobi identity for double Poisson brackets was found by T.~Schedler in \cite{schedler2009poisson}, see also Section 12.2 in \cite{rubtsov2023lectures}. The part about the classical Yang-Baxter equation essentially follows from Proposition 3.17 from \cite{fairon2023around}. We will check that \eqref{f33} is equivalent to \eqref{f2} in a slightly different context in Section \ref{sec1}. However, the computation to be performed remains valid in the current general case as well.
    
\end{remark}

\begin{theorem}\label{th1}
Let $\{-,-\}:\O(A)\otimes\O(A)\longrightarrow\O(A)$ be a di-twisted (not necessarily Poisson) bracket on $\O(A)$. Define $\{\!\!\{-,-\}\!\!\}_{\id},\{\!\!\{-,-\}\!\!\}_{(12)}:A\otimes A\longrightarrow A\otimes A\otimes \s$ as
\begin{align}
    \{a\otimes\mathds{1}\otimes \id_1,b\otimes\mathds{1}\otimes \id_1\}=\{\!\!\{a,b\}\!\!\}_{\id}\otimes \id_2+\{\!\!\{a,b\}\!\!\}_{(12)}\otimes (12)
\end{align}
for any $a,b\in A$. Then $\{\!\!\{-,-\}\!\!\}_{\id}$ is a right double bracket on $A$ and $\{\!\!\{-,-\}\!\!\}_{(12)}$ is a double bracket on $A$ in the sense of Definition \ref{def15}. Moreover, the assignment
\begin{align}
    \{-,-\}\mapsto \Big(\{\!\!\{-,-\}\!\!\}_{\id},\ \{\!\!\{-,-\}\!\!\}_{(12)}\Big)
\end{align}
defines a bijection between di-twisted brackets on $\O(A)$ and pairs of double and right double brackets on $A$, which maps di-twisted Poisson brackets to coupled double Poisson brackets.
\end{theorem}
We will prove this theorem in the subsequent section.

The next corollary states that coupled double Poisson brackets satisfy the Kontsevich-Rosenberg principle, and immediately follows from Proposition \ref{prop3} and Theorem \ref{th1}.

\begin{corollary}\label{cor_KR}
    Let $A$ be an associative algebra and let $\{\!\!\{-,-\}\!\!\}_{\id}$, $\{\!\!\{-,-\}\!\!\}_{(12)}$ be a pair of coupled double Poisson brackets on $A$. For each $N=1,2,\dots$, the formula
\begin{equation}
\{a_{ij},b_{kl}\}_N:=\big(\{\!\!\{a,b\}\!\!\}'_{\id}\big)_{ij}\big(\{\!\!\{a,b\}\!\!\}''_{\id}\big)_{kl}\tr\Big(\{\!\!\{a,b\}\!\!\}'''_{\id}\Big)+\big(\{\!\!\{a,b\}\!\!\}'_{(12)}\big)_{kj}\big(\{\!\!\{a,b\}\!\!\}''_{(12)}\big)_{il}\tr\Big(\{\!\!\{a,b\}\!\!\}'''_{(12)}\Big),
\end{equation}
gives rise to a $\GL$-invariant Poisson bracket $\{-,-\}_N$ on the commutative algebra $\orep$.
\end{corollary}

The next corollary shows that it is enough to check the complicated coupled double Poisson relations only for generators in $A$. We will prove this corollary in the next section.

\begin{corollary}\label{cor1}
    Let $\mathcal{G}\subset A$ be a system of generators. Let $\{\!\!\{-,-\}\!\!\}_{\id}$ and $\{\!\!\{-,-\}\!\!\}_{(12)}$ be a right double bracket and a double bracket in the sense of Definition \ref{def15}. If they satisfy relations 1), 2), 3) from Definition \ref{def16} for any $a,b,c\in\mathcal{G}$, then they satisfy these relations for all $a,b,c\in A$.
\end{corollary}

\section{Proof of Theorem \ref{th1} and Corollary \ref{cor1}}\label{sec_proof}

\begin{lemma}\label{lemma1}
    The maps $\{\!\!\{-,-\}\!\!\}_{\id},\{\!\!\{-,-\}\!\!\}_{(12)}:A\otimes A\longrightarrow A\otimes A\otimes \s$ defined in Theorem \ref{th1} are a right double bracket and a double bracket respectively.
\end{lemma}
\begin{proof}
   From the skew-symmetry for the di-twisted bracket $\{-,-\}$ we obtain
\begin{align}
    \{\!\!\{a,b\}\!\!\}_{\id}\otimes \id+\{\!\!\{a,b\}\!\!\}_{(12)}\otimes (12)=-(12)\Big[\{\!\!\{b,a\}\!\!\}_{\id}\otimes \id+\{\!\!\{b,a\}\!\!\}_{(12)}\otimes (12)\Big](12)
\end{align}
Hence,
\begin{align}
    \{\!\!\{a,b\}\!\!\}_{\id}&=-(12)\{\!\!\{b,a\}\!\!\}_{\id},\label{f93}\\
    \{\!\!\{a,b\}\!\!\}_{(12)}&=-(12)\{\!\!\{b,a\}\!\!\}_{(12)}\label{f94}
\end{align}

Since $\{-,\b\}$ and $\pi$ commute, one has
\begin{align}
    \{\pi(a\otimes b\otimes \mathds{1}\otimes (12)),c\otimes \mathds{1}\otimes \id_1\}&=\pi\{a\otimes b\otimes \mathds{1}\otimes (12),c\otimes \mathds{1}\otimes \id_1\}\\
    &=\pi\Big[\{a\otimes b\otimes \mathds{1}\otimes \id_2,c\otimes \mathds{1}\otimes \id_1\}(12)\Big].
\end{align}
By the definition, the left-hand side equals
\begin{align}
    \{ab,c\}=\{\!\!\{ab,c\}\!\!\}_{\id}\otimes \id+\{\!\!\{ab,c\}\!\!\}_{(12)}\otimes (12).
\end{align}
The right-hand side can be computed with the help of the Leibniz rule for the di-twisted Poisson bracket $\{-,-\}$, which gives
\begin{align}
    \hspace{15pt}&\hspace{-15pt}\pi\Big[(23)\{a,c\}(b\otimes \mathds{1}\otimes \id_1)(23)(12)+(a\otimes \mathds{1}\otimes \id_1)\{b,c\}(12)\Big]\\
    &=\pi\Big[\{\!\!\{a,c\}\!\!\}_{\id}'\otimes b\otimes \{\!\!\{a,c\}\!\!\}''_{\id}\otimes \{\!\!\{a,c\}\!\!\}_{\id}'''\otimes (12)\\
    &\hspace{20pt}+\{\!\!\{a,c\}\!\!\}_{(12)}'\otimes b\otimes \{\!\!\{a,c\}\!\!\}''_{(12)}\otimes \{\!\!\{a,c\}\!\!\}_{(12)}'''\otimes \overbrace{(23)(12)(23)(12)}^{(123)}\\
    &\hspace{20pt}+a\otimes \{\!\!\{b,c\}\!\!\}_{\id}'\otimes \{\!\!\{b,c\}\!\!\}''_{\id}\otimes \{\!\!\{b,c\}\!\!\}_{\id}'''\otimes (12)+a\otimes \{\!\!\{b,c\}\!\!\}_{(12)}'\otimes \{\!\!\{b,c\}\!\!\}''_{(12)}\otimes \{\!\!\{b,c\}\!\!\}_{(12)}'''\otimes\overbrace{(23)(12)}^{(132)} \Big]\\
    &=\{\!\!\{a,c\}\!\!\}_{\id}' b\otimes \{\!\!\{a,c\}\!\!\}''_{\id}\otimes \{\!\!\{a,c\}\!\!\}_{\id}'''\otimes \id_2+\{\!\!\{a,c\}\!\!\}_{(12)}' b\otimes \{\!\!\{a,c\}\!\!\}''_{(12)}\otimes \{\!\!\{a,c\}\!\!\}_{(12)}'''\otimes (12)\\
    &\pushright{+a \{\!\!\{b,c\}\!\!\}_{\id}'\otimes \{\!\!\{b,c\}\!\!\}''_{\id}\otimes \{\!\!\{b,c\}\!\!\}_{\id}'''\otimes \id_2+ \{\!\!\{b,c\}\!\!\}_{(12)}'\otimes a\{\!\!\{b,c\}\!\!\}''_{(12)}\otimes \{\!\!\{b,c\}\!\!\}_{(12)}'''\otimes (12)}.
\end{align}
Hence, we obtain 
\begin{align}
    \{\!\!\{ab,c\}\!\!\}_{\id}&=\{\!\!\{a,c\}\!\!\}_{\id}' b\otimes \{\!\!\{a,c\}\!\!\}''_{\id}\otimes \{\!\!\{a,c\}\!\!\}_{\id}'''+a \{\!\!\{b,c\}\!\!\}_{\id}'\otimes \{\!\!\{b,c\}\!\!\}''_{\id}\otimes \{\!\!\{b,c\}\!\!\}_{\id}''',\\
    \{\!\!\{ab,c\}\!\!\}_{(12)}&=\{\!\!\{a,c\}\!\!\}_{(12)}' b\otimes \{\!\!\{a,c\}\!\!\}''_{(12)}\otimes \{\!\!\{a,c\}\!\!\}_{(12)}'''+\{\!\!\{b,c\}\!\!\}_{(12)}'\otimes a\{\!\!\{b,c\}\!\!\}''_{(12)}\otimes \{\!\!\{b,c\}\!\!\}_{(12)}'''.
\end{align}
which can be rewritten as follows using the inner and outer bimodule structures on $A^{\otimes 2}$
\begin{align}
    \{\!\!\{ab,c\}\!\!\}_{\id}&=a\{\!\!\{b,c\}\!\!\}_{\id}+\{\!\!\{a,c\}\!\!\}_{\id}*b,\\
    \{\!\!\{ab,c\}\!\!\}_{(12)}&=a*\{\!\!\{b,c\}\!\!\}_{(12)}+\{\!\!\{a,c\}\!\!\}_{(12)}*b.
\end{align}
Note that by \eqref{f93},\eqref{f94} we also obtain
\begin{align}
    \{\!\!\{a,bc\}\!\!\}_{\id}&=\{\!\!\{a,b\}\!\!\}_{\id}c+b*\{\!\!\{a,c\}\!\!\}_{\id},\\
    \{\!\!\{a,bc\}\!\!\}_{(12)}&=\{\!\!\{a,b\}\!\!\}_{(12)}c+b\{\!\!\{a,c\}\!\!\}_{(12)}.
\end{align}

\end{proof}

We will need the following notation
\begin{align}
    &\Jac^{\id_3}(a,b,c):=\Jac^{\id_3}_{\id}(a,b,c)+(123)\Jac^{\id_3}_{\id,(12)}(a,b,c)\\[10pt]
    &\mathbbnew{Jac}^{\id_3}_{\id}(a,b,c):=\Big\{\!\!\!\Big\{a,\{\!\!\{b,c\}\!\!\}_{\id}\Big\}\!\!\!\Big\}_{\id}^{\text{left}}+(123)\Big\{\!\!\!\Big\{b,\{\!\!\{c,a\}\!\!\}_{\id}\Big\}\!\!\!\Big\}_{\id}^{\text{left}}+(132)\Big\{\!\!\!\Big\{c,\{\!\!\{a,b\}\!\!\}_{\id}\Big\}\!\!\!\Big\}_{\id}^{\text{left}}\\[7pt]
&\hspace{90pt}+(12)\Bigg[\Big\{\!\!\!\Big\{a,\{\!\!\{b,c\}\!\!\}_{\id}\Big\}\!\!\!\Big\}_{\id}^{\text{right}}+(132)\Big\{\!\!\!\Big\{b,\{\!\!\{c,a\}\!\!\}_{\id}\Big\}\!\!\!\Big\}_{\id}^{\text{right}}+(123)\Big\{\!\!\!\Big\{c,\{\!\!\{a,b\}\!\!\}_{\id}\Big\}\!\!\!\Big\}_{\id}^{\text{right}}\Bigg]\\[7pt]
&\hspace{90pt}+(123)\Bigg[\Big\{\!\!\!\Big\{a,\{\!\!\{b,c\}\!\!\}_{\id}\Big\}\!\!\!\Big\}_{\id}^{\text{aux}}+(123)\Big\{\!\!\!\Big\{b,\{\!\!\{c,a\}\!\!\}_{\id}\Big\}\!\!\!\Big\}_{\id}^{\text{aux}}+(132)\Big\{\!\!\!\Big\{c,\{\!\!\{a,b\}\!\!\}_{\id}\Big\}\!\!\!\Big\}_{\id}^{\text{aux}}\Bigg],\\[25pt]
    &\mathbbnew{Jac}^{\id_3}_{\id,(12)}(a,b,c):=\Big\{\!\!\!\Big\{a,\{\!\!\{b,c\}\!\!\}_{\id}\Big\}\!\!\!\Big\}_{(12)}^{\text{aux}}+(123)\Big\{\!\!\!\Big\{b,\{\!\!\{c,a\}\!\!\}_{\id}\Big\}\!\!\!\Big\}_{(12)}^{\text{aux}}+(132)\Big\{\!\!\!\Big\{c,\{\!\!\{a,b\}\!\!\}_{\id}\Big\}\!\!\!\Big\}_{(12)}^{\text{aux}}\\[25pt]
    &\Jac^{12}(a,b,c):=\Big\{\!\!\!\Big\{a,\{\!\!\{b,c\}\!\!\}_{\id}\Big\}\!\!\!\Big\}_{(12)}^{\text{left}}+(13)\Big\{\!\!\!\Big\{b,\{\!\!\{c,a\}\!\!\}_{\id}\Big\}\!\!\!\Big\}_{(12)}^{\text{right}}+(132)\Big\{\!\!\!\Big\{c,\{\!\!\{a,b\}\!\!\}_{(12)}\Big\}\!\!\!\Big\}_{\id}^{\text{left}}\\[7pt]
&\hspace{90pt}(23)\Big\{\!\!\!\Big\{c,\{\!\!\{a,b\}\!\!\}_{(12)}\Big\}\!\!\!\Big\}_{\id}^{\text{right}}+\Big\{\!\!\!\Big\{c,\{\!\!\{a,b\}\!\!\}_{(12)}\Big\}\!\!\!\Big\}_{\id}^{\text{aux}}+\Big\{\!\!\!\Big\{c,\{\!\!\{a,b\}\!\!\}_{(12)}\Big\}\!\!\!\Big\}_{(12)}^{\text{aux}}\\[25pt]
    &\Jac^{13}(a,b,c):=(12)\Big\{\!\!\!\Big\{a,\{\!\!\{b,c\}\!\!\}_{\id}\Big\}\!\!\!\Big\}_{(12)}^{\text{right}}+(123)\Big\{\!\!\!\Big\{b,\{\!\!\{c,a\}\!\!\}_{(12)}\Big\}\!\!\!\Big\}_{\id}^{\text{left}}+(13)\Big\{\!\!\!\Big\{b,\{\!\!\{c,a\}\!\!\}_{(12)}\Big\}\!\!\!\Big\}_{\id}^{\text{right}}\\[7pt]
&\hspace{90pt}+(132)\Big\{\!\!\!\Big\{b,\{\!\!\{c,a\}\!\!\}_{(12)}\Big\}\!\!\!\Big\}_{\id}^{\text{aux}}+(132)\Big\{\!\!\!\Big\{b,\{\!\!\{c,a\}\!\!\}_{(12)}\Big\}\!\!\!\Big\}_{(12)}^{\text{aux}}+(132)\Big\{\!\!\!\Big\{c,\{\!\!\{a,b\}\!\!\}_{\id}\Big\}\!\!\!\Big\}_{(12)}^{\text{left}}\\[25pt]
    &\Jac^{23}(a,b,c):=\Big\{\!\!\!\Big\{a,\{\!\!\{b,c\}\!\!\}_{(12)}\Big\}\!\!\!\Big\}_{\id}^{\text{left}}+(12)\Big\{\!\!\!\Big\{a,\{\!\!\{b,c\}\!\!\}_{(12)}\Big\}\!\!\!\Big\}_{\id}^{\text{right}}+(123)\Big\{\!\!\!\Big\{a,\{\!\!\{b,c\}\!\!\}_{(12)}\Big\}\!\!\!\Big\}_{\id}^{\text{aux}}\\
&\hspace{90pt}+(123)\Big\{\!\!\!\Big\{a,\{\!\!\{b,c\}\!\!\}_{(12)}\Big\}\!\!\!\Big\}_{(12)}^{\text{aux}}+(123)\Big\{\!\!\!\Big\{b,\{\!\!\{c,a\}\!\!\}_{\id}\Big\}\!\!\!\Big\}_{(12)}^{\text{left}}+(23)\Big\{\!\!\!\Big\{c,\{\!\!\{a,b\}\!\!\}_{\id}\Big\}\!\!\!\Big\}_{(12)}^{\text{right}}
\\[25pt]
    &\mathbbnew{Jac}^{123}(a,b,c):=\Big\{\!\!\!\Big\{a,\{\!\!\{b,c\}\!\!\}_{(12)}\Big\}\!\!\!\Big\}_{(12)}^{\text{left}}+(123)\Big\{\!\!\!\Big\{b,\{\!\!\{c,a\}\!\!\}_{(12)}\Big\}\!\!\!\Big\}_{(12)}^{\text{left}}+(132)\Big\{\!\!\!\Big\{c,\{\!\!\{a,b\}\!\!\}_{(12)}\Big\}\!\!\!\Big\}_{(12)}^{\text{left}}\\[25pt]
&\Jac^{132}(a,b,c):=(12)\Big\{\!\!\!\Big\{a,\{\!\!\{b,c\}\!\!\}_{(12)}\Big\}\!\!\!\Big\}_{(12)}^{\text{right}}+(13)\Big\{\!\!\!\Big\{b,\{\!\!\{c,a\}\!\!\}_{(12)}\Big\}\!\!\!\Big\}_{(12)}^{\text{right}}+(23)\Big\{\!\!\!\Big\{c,\{\!\!\{a,b\}\!\!\}_{(12)}\Big\}\!\!\!\Big\}_{(12)}^{\text{right}}.
\end{align}

\begin{lemma}\label{lemma2}
    Using notation from Theorem \ref{th1}, one has 
    \begin{align}\label{f99}
    &\bigl\{ a,\{b,c\}\bigr\}+\operatorname{Ad}(123) \bigl\{ b,\{ c,a\}\bigr\}+\operatorname{Ad}(132) \bigl\{ c,\{ a,b\}\bigr\}\\[7pt]
    &\hspace{40pt}=\Jac^{\id_3}(a,b,c)\otimes \id_3+\Jac^{12}(a,b,c)\otimes (12)+\Jac^{13}(a,b,c)\otimes (13)+\Jac^{23}(a,b,c)\otimes (23)\\
    &\pushright{+\Jac^{123}(a,b,c)\otimes (123)+\Jac^{132}(a,b,c)\otimes (132)},
\end{align}
where on the left-hand side we write $a\in A$ for $a\otimes\mathds{1}\otimes\id_1\in\O(A)_1$ for brevity, and similarly for $b,c$. 
\end{lemma}
\begin{proof}

By definition, we have 
\begin{align}\label{f96}
    \bigl\{ a,\{b,c\}\bigr\}&=\Big\{a,\{\!\!\{b,c\}\!\!\}_{\id}\otimes \id+\{\!\!\{b,c\}\!\!\}_{(12)}\otimes (12)\Big\}\\
    &=\Big\{a,\{\!\!\{b,c\}\!\!\}_{\id}\otimes \id_2\Big\}+\Big\{a,\{\!\!\{b,c\}\!\!\}_{(12)}\otimes \id_2\Big\}(23).
\end{align}
By the Leibniz rule, we compute
\begin{align}
    \Big\{a,\{\!\!\{b,c\}\!\!\}_{\id}\otimes \id_2\Big\}&=\Big\{a,\{\!\!\{b,c\}\!\!\}'_{\id}\Big\}\Big(\{\!\!\{b,c\}\!\!\}''_{\id}\otimes \{\!\!\{b,c\}\!\!\}'''_{\id}\otimes \id_1\Big)\\
    &\hspace{30pt}+(12)\Big(\{\!\!\{b,c\}\!\!\}'_{\id}\otimes\mathds{1}\otimes \id_1\Big)\Big\{a,\{\!\!\{b,c\}\!\!\}''_{\id}\Big\}\Big(\{\!\!\{b,c\}\!\!\}'''_{\id}\otimes \id_0\Big)(12) \label{f95}\\
    &\hspace{30pt}+(123)\Big(\{\!\!\{b,c\}\!\!\}'_{\id}\otimes\{\!\!\{b,c\}\!\!\}''_{\id}\otimes \mathds{1}\otimes \id_2\Big)\Big\{a,\{\!\!\{b,c\}\!\!\}'''_{\id}\otimes \id_0\Big\}(132)
\end{align}

From a simple computation
\begin{align}
    \Big\{a,\overline{f_1}\cdot\ldots\cdot\overline{f_r}\otimes \id_0\Big\}&=\sum\limits_{k=1}^r(\overline{f_1}\cdot\ldots\cdot\widehat{\overline{f_k}}\cdot\ldots\cdot\overline{f_r}\otimes \id_0)\Big\{a,\overline{f_k}\otimes \id_0\Big\}\\
    &=-\sum\limits_{k=1}^r(\overline{f_1}\cdot\ldots\cdot\widehat{\overline{f_k}}\cdot\ldots\cdot\overline{f_r}\otimes \id_0)\Big\{\pi(f_k\otimes\mathds{1}\otimes \id_1),a\Big\}\\
    &=\sum\limits_{k=1}^r(\overline{f_1}\cdot\ldots\cdot\widehat{\overline{f_k}}\cdot\ldots\cdot\overline{f_r}\otimes \id_0)\pi\Big[(12)\{a,f_k\}(12)\Big]\\
    &=\sum\limits_{k=1}^r(\overline{f_1}\cdot\ldots\cdot\widehat{\overline{f_k}}\cdot\ldots\cdot\overline{f_r}\otimes \id_0)\pi\Big[\{\!\!\{a,f_k\}\!\!\}^{\circ}_{\id}\otimes \id_2+\{\!\!\{a,f_k\}\!\!\}^{\circ}_{(12)}\otimes (12)\Big]\\
    &=\sum\limits_{k=1}^r(\overline{f_1}\cdot\ldots\cdot\widehat{\overline{f_k}}\cdot\ldots\cdot\overline{f_r}\otimes \id_0)\Big(\{\!\!\{a,f_k\}\!\!\}'_{\id}\otimes \overline{\{\!\!\{a,f_k\}\!\!\}''_{\id}}\cdot \{\!\!\{a,f_k\}\!\!\}'''_{\id}\otimes \id_1\\
    &\pushright{+\{\!\!\{a,f_k\}\!\!\}''_{(12)}\{\!\!\{a,f_k\}\!\!\}_{(12)}'\otimes \{\!\!\{a,f_k\}\!\!\}'''_{(12)}\otimes \id_1}\Big)\\
    &=\Big[\{a,\overline{f_1}\cdot\ldots\cdot\overline{f_r}\}_{\id}+\{a,\overline{f_1}\cdot\ldots\cdot\overline{f_r}\}_{(12)}\Big]\otimes \id_1
    \end{align}
    one deduces that $\{a,f\otimes \id_0\}=\Big(\{a,f\}_{\id}+\{a,f\}_{(12)}\Big)\otimes \id_1$ for any $f\in\s$, particularly,
    \begin{align}
    \Big\{a,\{\!\!\{b,c\}\!\!\}'''_{\id}\otimes \id_0\Big\}=\Big\{a,\{\!\!\{b,c\}\!\!\}'''_{\id}\Big\}_{\id}\otimes \id_1+\Big\{a,\{\!\!\{b,c\}\!\!\}'''_{\id}\Big\}_{(12)}\otimes \id_1,    
    \end{align}
    so we can continue \eqref{f95}
    \begin{align}
    \hspace{17pt}&\hspace{-17pt}\Big\{a,\{\!\!\{b,c\}\!\!\}_{\id}\otimes \id_2\Big\}\\
&=\Big[\{\!\!\{a,\{\!\!\{b,c\}\!\!\}'_{\id}\}\!\!\}_{\id}\otimes \id+\{\!\!\{a,\{\!\!\{b,c\}\!\!\}'_{\id}\}\!\!\}_{(12)}\otimes (12)\Big]\Big(\{\!\!\{b,c\}\!\!\}''_{\id}\otimes \{\!\!\{b,c\}\!\!\}'''_{\id}\otimes \id_1\Big)\\
    &\hspace{15pt}+(12)\Big(\{\!\!\{b,c\}\!\!\}'_{\id}\otimes\mathds{1}\otimes \id_1\Big)\Big[\{\!\!\{a,\{\!\!\{b,c\}\!\!\}''_{\id}\}\!\!\}_{\id}\otimes \id_2+\{\!\!\{a,\{\!\!\{b,c\}\!\!\}''_{\id}\}\!\!\}_{(12)}\otimes (12)\Big]\Big(\{\!\!\{b,c\}\!\!\}'''_{\id}\otimes \id_0\Big)(12)\\
    &\pushright{(123)\Big(\{\!\!\{b,c\}\!\!\}'_{\id}\otimes\{\!\!\{b,c\}\!\!\}''_{\id}\otimes \mathds{1}\otimes \id_2\Big)\Big(\Big\{a,\{\!\!\{b,c\}\!\!\}'''_{\id}\Big\}_{\id}\otimes \id_1+\Big\{a,\{\!\!\{b,c\}\!\!\}'''_{\id}\Big\}_{(12)}\otimes \id_1\Big)(132)}\\
&=\{\!\!\{a,\{\!\!\{b,c\}\!\!\}'_{\id}\}\!\!\}'_{\id}\otimes \{\!\!\{a,\{\!\!\{b,c\}\!\!\}'_{\id}\}\!\!\}''_{\id}\otimes \{\!\!\{b,c\}\!\!\}''_{\id}\otimes \{\!\!\{a,\{\!\!\{b,c\}\!\!\}'_{\id}\}\!\!\}'''_{\id}\cdot\{\!\!\{b,c\}\!\!\}'''_{\id}\otimes \id_3\\[7pt]
    &\hspace{50pt}+\{\!\!\{a,\{\!\!\{b,c\}\!\!\}'_{\id}\}\!\!\}'_{(12)}\otimes \{\!\!\{a,\{\!\!\{b,c\}\!\!\}'_{\id}\}\!\!\}''_{(12)}\otimes \{\!\!\{b,c\}\!\!\}''_{\id}\otimes \{\!\!\{a,\{\!\!\{b,c\}\!\!\}'_{\id}\}\!\!\}'''_{(12)}\cdot\{\!\!\{b,c\}\!\!\}'''_{\id}\otimes (12)\\[7pt]
    &\hspace{50pt}+\{\!\!\{a,\{\!\!\{b,c\}\!\!\}''_{\id}\}\!\!\}'_{\id}\otimes \{\!\!\{b,c\}\!\!\}'_{\id}\otimes\{\!\!\{a,\{\!\!\{b,c\}\!\!\}''_{\id}\}\!\!\}''_{\id}\otimes \{\!\!\{a,\{\!\!\{b,c\}\!\!\}''_{\id}\}\!\!\}'''_{\id}\cdot\{\!\!\{b,c\}\!\!\}'''_{\id}\otimes \id_3\\[7pt]
    &\hspace{50pt}+ \{\!\!\{a,\{\!\!\{b,c\}\!\!\}''_{\id}\}\!\!\}'_{(12)}\otimes \{\!\!\{b,c\}\!\!\}'_{\id}\otimes\{\!\!\{a,\{\!\!\{b,c\}\!\!\}''_{\id}\}\!\!\}''_{(12)}\otimes \{\!\!\{a,\{\!\!\{b,c\}\!\!\}''_{\id}\}\!\!\}'''_{(12)}\cdot \{\!\!\{b,c\}\!\!\}'''_{\id}\otimes (13)\\[7pt]
    &\hspace{50pt}+\Big\{a,\{\!\!\{b,c\}\!\!\}'''_{\id}\Big\}'_{\id}\otimes\{\!\!\{b,c\}\!\!\}'_{\id}\otimes\{\!\!\{b,c\}\!\!\}''_{\id}\otimes \Big\{a,\{\!\!\{b,c\}\!\!\}'''_{\id}\Big\}''_{\id}\otimes \id_3\\[7pt]
    &\hspace{50pt}+\Big\{a,\{\!\!\{b,c\}\!\!\}'''_{\id}\Big\}'_{(12)}\otimes \{\!\!\{b,c\}\!\!\}'_{\id}\otimes\{\!\!\{b,c\}\!\!\}''_{\id}\otimes \Big\{a,\{\!\!\{b,c\}\!\!\}'''_{\id}\Big\}''_{(12)}\otimes \id_3.
\end{align}

Then 
\begin{align}
\Big\{a,\{\!\!\{b,c\}\!\!\}_{\id}\otimes \id_2\Big\}\hspace{30pt}&\hspace{-30pt}=\Big\{\!\!\!\Big\{a,\{\!\!\{b,c\}\!\!\}_{\id}\Big\}\!\!\!\Big\}_{\id}^{\text{left}}\otimes \id_3+\Big\{\!\!\!\Big\{a,\{\!\!\{b,c\}\!\!\}_{\id}\Big\}\!\!\!\Big\}_{(12)}^{\text{left}}\otimes (12)\\[7pt]
&\hspace{10pt}+\Bigg[(12)\Big\{\!\!\!\Big\{a,\{\!\!\{b,c\}\!\!\}_{\id}\Big\}\!\!\!\Big\}_{\id}^{\text{right}}\Bigg]\otimes\id_3+\Bigg[(12)\Big\{\!\!\!\Big\{a,\{\!\!\{b,c\}\!\!\}_{\id}\Big\}\!\!\!\Big\}_{(12)}^{\text{right}}\Bigg]\otimes (13)\\[7pt]
&\hspace{10pt}+\Bigg[(123)\Big\{\!\!\!\Big\{a,\{\!\!\{b,c\}\!\!\}_{\id}\Big\}\!\!\!\Big\}_{\id}^{\text{aux}}\Bigg]\otimes\id_3+\Bigg[(123)\Big\{\!\!\!\Big\{a,\{\!\!\{b,c\}\!\!\}_{\id}\Big\}\!\!\!\Big\}_{(12)}^{\text{aux}}\Bigg]\otimes\id_3.
\end{align}

Similarly,
\begin{align}
    \Big\{a,\{\!\!\{b,c\}\!\!\}_{(12)}\otimes \id_2\Big\}(23)\hspace{20pt}&\hspace{-20pt}=\Big\{\!\!\!\Big\{a,\{\!\!\{b,c\}\!\!\}_{(12)}\Big\}\!\!\!\Big\}_{\id}^{\text{left}}\otimes (23)+\Big\{\!\!\!\Big\{a,\{\!\!\{b,c\}\!\!\}_{(12)}\Big\}\!\!\!\Big\}_{(12)}^{\text{left}}\otimes (123)\\[7pt]
&\hspace{0pt}+\Bigg[(12)\Big\{\!\!\!\Big\{a,\{\!\!\{b,c\}\!\!\}_{(12)}\Big\}\!\!\!\Big\}_{\id}^{\text{right}}\Bigg]\otimes(23)+\Bigg[(12)\Big\{\!\!\!\Big\{a,\{\!\!\{b,c\}\!\!\}_{(12)}\Big\}\!\!\!\Big\}_{(12)}^{\text{right}}\Bigg]\otimes (132)\\[7pt]
&\hspace{0pt}+\Bigg[(123)\Big\{\!\!\!\Big\{a,\{\!\!\{b,c\}\!\!\}_{(12)}\Big\}\!\!\!\Big\}_{\id}^{\text{aux}}\Bigg]\otimes(23)+\Bigg[(123)\Big\{\!\!\!\Big\{a,\{\!\!\{b,c\}\!\!\}_{(12)}\Big\}\!\!\!\Big\}_{(12)}^{\text{aux}}\Bigg]\otimes(23).
\end{align}
Then by \eqref{f96} one has
\begin{align}
\bigl\{ a,\{b,c\}\bigr\}\hspace{17pt}&\hspace{-17pt}=\Bigg[\Big\{\!\!\!\Big\{a,\{\!\!\{b,c\}\!\!\}_{\id}\Big\}\!\!\!\Big\}_{\id}^{\text{left}}+(12)\Big\{\!\!\!\Big\{a,\{\!\!\{b,c\}\!\!\}_{\id}\Big\}\!\!\!\Big\}_{\id}^{\text{right}}\\[7pt]
&\hspace{130pt}+(123)\Big\{\!\!\!\Big\{a,\{\!\!\{b,c\}\!\!\}_{\id}\Big\}\!\!\!\Big\}_{\id}^{\text{aux}}+(123)\Big\{\!\!\!\Big\{a,\{\!\!\{b,c\}\!\!\}_{\id}\Big\}\!\!\!\Big\}_{(12)}^{\text{aux}}\Bigg]\otimes \id_3\\[7pt]
&+\Big\{\!\!\!\Big\{a,\{\!\!\{b,c\}\!\!\}_{\id}\Big\}\!\!\!\Big\}_{(12)}^{\text{left}}\otimes (12)+\Bigg[(12)\Big\{\!\!\!\Big\{a,\{\!\!\{b,c\}\!\!\}_{\id}\Big\}\!\!\!\Big\}_{(12)}^{\text{right}}\Bigg]\otimes (13)\\[7pt]
&+\Bigg[\Big\{\!\!\!\Big\{a,\{\!\!\{b,c\}\!\!\}_{(12)}\Big\}\!\!\!\Big\}_{\id}^{\text{left}}+(12)\Big\{\!\!\!\Big\{a,\{\!\!\{b,c\}\!\!\}_{(12)}\Big\}\!\!\!\Big\}_{\id}^{\text{right}}\\
&\hspace{130pt}+(123)\Big\{\!\!\!\Big\{a,\{\!\!\{b,c\}\!\!\}_{(12)}\Big\}\!\!\!\Big\}_{\id}^{\text{aux}}+(123)\Big\{\!\!\!\Big\{a,\{\!\!\{b,c\}\!\!\}_{(12)}\Big\}\!\!\!\Big\}_{(12)}^{\text{aux}}\Bigg]\otimes (23)\\[7pt]
&+\Big\{\!\!\!\Big\{a,\{\!\!\{b,c\}\!\!\}_{(12)}\Big\}\!\!\!\Big\}_{(12)}^{\text{left}}\otimes (123)+\Bigg[(12)\Big\{\!\!\!\Big\{a,\{\!\!\{b,c\}\!\!\}_{(12)}\Big\}\!\!\!\Big\}_{(12)}^{\text{right}}\Bigg]\otimes (132)
\end{align}
and similarly
\begin{align}
\hspace{30pt}&\hspace{-30pt}\operatorname{Ad}(123)\bigl\{ b,\{c,a\}\bigr\}\\[7pt]
&=\Bigg[(123)\Big\{\!\!\!\Big\{b,\{\!\!\{c,a\}\!\!\}_{\id}\Big\}\!\!\!\Big\}_{\id}^{\text{left}}+(13)\Big\{\!\!\!\Big\{b,\{\!\!\{c,a\}\!\!\}_{\id}\Big\}\!\!\!\Big\}_{\id}^{\text{right}}\\[7pt]
&\hspace{90pt}+(132)\Big\{\!\!\!\Big\{b,\{\!\!\{c,a\}\!\!\}_{\id}\Big\}\!\!\!\Big\}_{\id}^{\text{aux}}+(132)\Big\{\!\!\!\Big\{b,\{\!\!\{c,a\}\!\!\}_{\id}\Big\}\!\!\!\Big\}_{(12)}^{\text{aux}}\Bigg]\otimes \id_3\\[7pt]
&+\Bigg[(123)\Big\{\!\!\!\Big\{b,\{\!\!\{c,a\}\!\!\}_{\id}\Big\}\!\!\!\Big\}_{(12)}^{\text{left}}\Bigg]\otimes \underbrace{(123)(12)(132)}_{(23)}+\Bigg[(13)\Big\{\!\!\!\Big\{b,\{\!\!\{c,a\}\!\!\}_{\id}\Big\}\!\!\!\Big\}_{(12)}^{\text{right}}\Bigg]\otimes \underbrace{(123)(13)(132)}_{(12)}\\[7pt]
&+\Bigg[(123)\Big\{\!\!\!\Big\{b,\{\!\!\{c,a\}\!\!\}_{(12)}\Big\}\!\!\!\Big\}_{\id}^{\text{left}}+(13)\Big\{\!\!\!\Big\{b,\{\!\!\{c,a\}\!\!\}_{(12)}\Big\}\!\!\!\Big\}_{\id}^{\text{right}}\\
&\hspace{90pt}+(132)\Big\{\!\!\!\Big\{b,\{\!\!\{c,a\}\!\!\}_{(12)}\Big\}\!\!\!\Big\}_{\id}^{\text{aux}}+(132)\Big\{\!\!\!\Big\{b,\{\!\!\{c,a\}\!\!\}_{(12)}\Big\}\!\!\!\Big\}_{(12)}^{\text{aux}}\Bigg]\otimes \underbrace{(123)(23)(132)}_{(13)}\\[7pt]
&+\Bigg[(123)\Big\{\!\!\!\Big\{b,\{\!\!\{c,a\}\!\!\}_{(12)}\Big\}\!\!\!\Big\}_{(12)}^{\text{left}}\Bigg]\otimes (123)+\Bigg[(13)\Big\{\!\!\!\Big\{b,\{\!\!\{c,a\}\!\!\}_{(12)}\Big\}\!\!\!\Big\}_{(12)}^{\text{right}}\Bigg]\otimes (132)
\end{align}
and
\begin{align}
\hspace{30pt}&\hspace{-30pt}\operatorname{Ad}(132)\bigl\{ c,\{a,b\}\bigr\}\\[7pt]
&=\Bigg[(132)\Big\{\!\!\!\Big\{c,\{\!\!\{a,b\}\!\!\}_{\id}\Big\}\!\!\!\Big\}_{\id}^{\text{left}}+(23)\Big\{\!\!\!\Big\{c,\{\!\!\{a,b\}\!\!\}_{\id}\Big\}\!\!\!\Big\}_{\id}^{\text{right}}\\[7pt]
&\hspace{90pt}+\Big\{\!\!\!\Big\{c,\{\!\!\{a,b\}\!\!\}_{\id}\Big\}\!\!\!\Big\}_{\id}^{\text{aux}}+\Big\{\!\!\!\Big\{c,\{\!\!\{a,b\}\!\!\}_{\id}\Big\}\!\!\!\Big\}_{(12)}^{\text{aux}}\Bigg]\otimes \id_3\\[7pt]
&+\Bigg[(132)\Big\{\!\!\!\Big\{c,\{\!\!\{a,b\}\!\!\}_{\id}\Big\}\!\!\!\Big\}_{(12)}^{\text{left}}\Bigg]\otimes \underbrace{(123)(23)(132)}_{(13)}+\Bigg[(23)\Big\{\!\!\!\Big\{c,\{\!\!\{a,b\}\!\!\}_{\id}\Big\}\!\!\!\Big\}_{(12)}^{\text{right}}\Bigg]\otimes \underbrace{(123)(12)(132)}_{(23)}\\[7pt]
&+\Bigg[(132)\Big\{\!\!\!\Big\{c,\{\!\!\{a,b\}\!\!\}_{(12)}\Big\}\!\!\!\Big\}_{\id}^{\text{left}}+(23)\Big\{\!\!\!\Big\{c,\{\!\!\{a,b\}\!\!\}_{(12)}\Big\}\!\!\!\Big\}_{\id}^{\text{right}}\\
&\hspace{90pt}+\Big\{\!\!\!\Big\{c,\{\!\!\{a,b\}\!\!\}_{(12)}\Big\}\!\!\!\Big\}_{\id}^{\text{aux}}+\Big\{\!\!\!\Big\{c,\{\!\!\{a,b\}\!\!\}_{(12)}\Big\}\!\!\!\Big\}_{(12)}^{\text{aux}}\Bigg]\otimes \underbrace{(123)(13)(132)}_{(12)}\\[7pt]
&+\Bigg[(132)\Big\{\!\!\!\Big\{c,\{\!\!\{a,b\}\!\!\}_{(12)}\Big\}\!\!\!\Big\}_{(12)}^{\text{left}}\Bigg]\otimes (123)+\Bigg[(23)\Big\{\!\!\!\Big\{c,\{\!\!\{a,b\}\!\!\}_{(12)}\Big\}\!\!\!\Big\}_{(12)}^{\text{right}}\Bigg]\otimes (132).
\end{align}
Finally, taking sum of these three expressions, one obtains \eqref{f99}.
    
\end{proof}

\subsection{Proof of Theorem \ref{th1}}
The inverse map to the one given in the formulation of the theorem assigns a di-twisted bracket $\{-,-\}$ on $\O(A)$ to a pair of a double bracket $\{\!\!\{-,-\}\!\!\}_{(12)}$ and a right double bracket $\{\!\!\{-,-\}\!\!\}_{\id}$ on $A$, and is given for any $\a=(a_1\otimes \ldots\otimes a_n)\otimes f\otimes u$ and $\b=(b_1\otimes \ldots\otimes b_m)\otimes g\otimes v$ by 
\begin{align}\label{f103}
    \{\a,\b\}=&\sum\limits_{l=1}^n\sum\limits_{r=1}^m a_1\otimes \ldots\otimes a_{l-1}\otimes \{\!\!\{a_l,b_r\}\!\!\}'_{\id}\otimes a_{l+1}\otimes \ldots\otimes a_n\otimes b_1\otimes\ldots\otimes b_{r-1}\otimes \{\!\!\{a_l,b_r\}\!\!\}''_{\id}\\
        &\pushright{\otimes b_{r+1}\otimes \ldots\otimes b_m\otimes \{\!\!\{a_l,b_r\}\!\!\}'''_{\id}\cdot f\cdot g\otimes u\times v}\\
            &+\sum\limits_{l=1}^n\sum\limits_{r=1}^m a_1\otimes \ldots\otimes a_{l-1}\otimes \{\!\!\{a_l,b_r\}\!\!\}'_{(12)}\otimes a_{l+1}\otimes \ldots\otimes a_n\otimes b_1\otimes\ldots\otimes b_{r-1}\otimes \{\!\!\{a_l,b_r\}\!\!\}''_{(12)}\\
                &\pushright{\otimes b_{r+1}\otimes \ldots\otimes b_m\otimes \{\!\!\{a_l,b_r\}\!\!\}'''_{(12)}\cdot f\cdot g\otimes (l\ n+r)(u\times v)}\\
                    &+\sum\limits_{l=1}^n  a_1\otimes \ldots\otimes a_{l-1}\otimes \{a_l,g\}'_{\id}\otimes a_{l+1}\otimes\ldots\otimes a_n\otimes b_1\otimes \ldots\otimes b_m\otimes \{a_l,g\}''_{\id}\cdot f\otimes u\times v\\
                        &+\sum\limits_{l=1}^n  a_1\otimes \ldots\otimes a_{l-1}\otimes \{a_l,g\}'_{(12)}\otimes a_{l+1}\otimes\ldots\otimes a_n\otimes b_1\otimes \ldots\otimes b_m\otimes \{a_l,g\}''_{(12)}\cdot f\otimes u\times v\\
                    &-\sum\limits_{r=1}^m a_1\otimes\ldots\otimes a_n\otimes b_1\otimes \ldots \otimes b_{r-1}\otimes \{b_r,f\}'_{\id}\otimes b_{r+1}\otimes \ldots \otimes b_m\otimes \{b_r,f\}''_{\id}\cdot g\otimes u\times v\\
                &-\sum\limits_{r=1}^m a_1\otimes\ldots\otimes a_n\otimes b_1\otimes \ldots \otimes b_{r-1}\otimes \{b_r,f\}'_{(12)}\otimes b_{r+1}\otimes \ldots \otimes b_m\otimes \{b_r,f\}''_{(12)}\cdot g\otimes u\times v\\
            &+a_1\otimes \ldots\otimes a_n\otimes b_1\otimes \ldots\otimes b_m\otimes \Big(\{f,g\}_{\id}+\{f,g\}_{(12)}\Big)\otimes u\times v.
  \end{align}
  
  Note that in the special case, when $\a=a\otimes\mathds{1}\otimes \id_1$ and $\b=b\otimes \mathds{1}\otimes \id_1$ one has
  \begin{equation}
      \{\a,\b\}=\{\!\!\{a,b\}\!\!\}_{\id}\otimes \id+\{\!\!\{a,b\}\!\!\}_{(12)}\otimes (12).
  \end{equation}
The lengthy formula \eqref{f103} is the obvious extension by the Leibniz rule of the special case above.

It is readily seen that the map $\{-,-\}:\O(A)\otimes \O(A)\longrightarrow\O(A)$ so defined is a skew-symmetric graded map of degree zero whose restriction to $\O(A)_n\otimes \O(A)_m$ is a morphism of $S(n)\times S(m)$-bimodules. It is straightforward to check that $\{\pi(\a),\b\}=\pi\{\a,\b\}$. It is an easy exercise to prove that $\{-,-\}$ satisfies the Leibniz rule, which completes the proof.

\begin{remark}\label{rem1}
    
Note that the expressions $\Jac^{\tau}(a,b,c)$, $\tau\in S(3)$, defined above, can be defined for any double and right double brackets, and these expressions satisfy the following ``skew-symmetry'' relations:
\begin{align}
    \Jac^{\tau}(\sigma(a_1,a_2,a_3))=(-1)^{\sigma}\sigma \Jac^{\sigma^{-1}\tau\sigma}(a_1,a_2,a_3), \hspace{20pt} \tau,\sigma\in S(3).
\end{align}
This follows from Theorem \ref{th1}, Lemma \ref{lemma2}, and the skew-symmetry of the di-twisted Jacobiator:
\begin{align}\label{f1}
    \Jac(\a_{\sigma^{-1}(1)},\a_{\sigma^{-1}(2)},\a_{\sigma^{-1}(3)})=(-1)^{\sigma}\operatorname{Ad}(\sigma^{|\a_1|,|\a_2|,|\a_3|})\Jac(\a_1,\a_2,\a_3).
\end{align}
\end{remark}
\subsection{Proof of Corollary \ref{cor1}}
    Consider the corresponding di-twisted bracket on $\O(A)$ provided by Theorem \ref{th1} and the di-twisted Jacobiator, which is the left-hand side of the Jacobi identity in Definition \ref{def2}. The Jacobiator is a di-twisted derivation in the last argument, i.e.,
    \begin{align}
        \Jac(\a,\b,\cc_1\cc_2)=\operatorname{Ad}((12)^{|\cc_1|,|\a|+|\b|})\Big[\cc_1\Jac(\a,\b,\cc_2)\Big]+\Jac(\a,\b,\cc_1)\cc_2,
    \end{align}
    skew-symmetric in the sense of \eqref{f1}, and satisfies
    \begin{align}
        \Jac(\pi(\a),\b,\cc)=\pi\Jac(\a,\b,\cc).
    \end{align}
    It is an easy exercise to prove these identities directly; however, there is an even simpler explanation of why they must hold. One way to convince yourself is to recall that $\O(A)$ equipped with the di-twisted bracket $\{-,-\}$ is a commutative algebra with a skew-symmetric bi-derivation in an appropriate category (of diagonal $\S$-bimodules), see \cite{ginzburg2010differentialoperatorsbvstructures} and \cite{fernández2025symplecticwheelgebrasnoncommutativegeometry}. The other way is to use Proposition \ref{prop3} and the following identity in $\orep$
    \begin{align}
        \Jac(\a,\b,\cc)_{\mathbf{i}\sqcup \mathbf{k}\sqcup \mathbf{p},\mathbf{j}\sqcup \mathbf{l}\sqcup \mathbf{q}}=\operatorname{Jac}(\a_{\mathbf{i},\mathbf{j}},\b_{\mathbf{k},\mathbf{l}},\cc_{\mathbf{p},\mathbf{q}}),
    \end{align}
    where $\operatorname{Jac}$ stands for the usual Jacobiator of the induced Poisson bracket $\{-,-\}_N$ in $\orep$.
    
    Hence, due to $\pi(a\otimes b\otimes f\otimes (12))=ab\otimes f\otimes \id_1$ for $a,b\in A$, $f\in\s$, the Jacobiator is zero if and only if it is zero on $\mathcal{G}\otimes \mathcal{G}\otimes \mathcal{G}\subset A\otimes A\otimes A\subset \O(A)\otimes \O(A)\otimes \O(A)$. Application of Lemma \ref{lemma2} and Remark \ref{rem1} completes the proof.

\section{Linear coupled double Poisson brackets and Poisson-left-pre-Lie algebras}\label{sec_linear}

In this section, we will restrict ourselves to the third case of Remark \ref{rem2}, i.e., we will work only with coupled double Poisson brackets such that $\{\!\!\{a,b\}\!\!\}_{\id}, \{\!\!\{a,b\}\!\!\}_{(12)}\in A\otimes A$.

Let $V$ be a vector space and denote by $\T(V)=\kk\oplus V\oplus V^{\otimes 2}\oplus \ldots$ the tensor algebra of $V$.
\begin{definition}\label{def1}
    A double/right double bracket $\br$ on $\T(V)$ is said to be \textit{linear} if for any $a,b\in V$ one has $\{\!\!\{a,b\}\!\!\}\in 1\otimes V\oplus V\otimes 1\subset \T(V)\otimes \T(V)$.   
\end{definition}

We refer the reader to \cite{pichereau2008double} and \cite{odesskii2013double} for linear double Poisson brackets.

\begin{proposition}[Proposition 10 in \cite{pichereau2008double}]\label{prop4}
    Any linear double Poisson bracket on $\T(V)$ is of the form
    \begin{align}
        \{\!\!\{a,b\}\!\!\}=1\otimes \mu(a,b)-\mu(b,a)\otimes 1,\hspace{20pt} a,b\in V
    \end{align}
    for an associative multiplication $\mu:V\otimes V\longrightarrow V$.
\end{proposition}

\begin{remark}\label{rem3}
    One can check that the double Jacobiator evaluated at $a,b,c\in V$ of the double bracket $\br$ defined by $\{\!\!\{a,b\}\!\!\}=1\otimes \mu(a,b)-\mu(b,a)\otimes 1$, for $a,b\in V$, equals 
\begin{align}
    1\otimes 1\otimes \Assoc(b,a,c)+1\otimes \Assoc(a,c,b)\otimes 1+\Assoc(c,b,a)\otimes 1\otimes 1,
\end{align}
where $\Assoc(a,b,c)=\mu\big(\mu(a,b),c\big)-\mu\big(a,\mu(b,c)\big)$ is the associator of $\mu$.
\end{remark}

Recall the definition of left pre-Lie algebras, see \cite{burde2006left}, or \cite{manchon2011short}, or \cite{bai2021introduction}.

\begin{definition}
    A left pre-Lie algebra is a vector space $V$ with a linear map $\mu:V\otimes V\longrightarrow V$ whose associator is symmetric in the first two arguments, i.e.,
    \begin{align}
        \mu\big(\mu(a,b),c\big)-\mu\big(a,\mu(b,c)\big)=\mu\big(\mu(b,a),c\big)-\mu\big(b,\mu(a,c)\big).
    \end{align}
\end{definition}

\begin{corollary}\label{cor2}
    Any linear right double bracket on $\T(V)$ is of the form
    \begin{align}
        \{\!\!\{a,b\}\!\!\}=1\otimes \mu(a,b)-\mu(b,a)\otimes 1,\hspace{20pt} a,b\in V
    \end{align}
    for a left pre-Lie algebra structure $\mu$ on $V$.
\end{corollary}
\begin{proof}
    Follows from Remark \ref{rem3} above and Lemma 3.12 in \cite{fairon2023around}.
\end{proof}

\begin{definition}[Poisson-left-pre-Lie algebras\footnote{Not to be confused with left pre-Poisson algebras introduced in \cite{aguiar2000pre}}]
    Let $V$ be a vector space and $\cdot$, $\{-,-\}$ be linear maps $V\otimes V\longrightarrow V$. We call $\big(V,\cdot, \{-,-\}\big)$ a \textit{Poisson-left-pre-Lie algebra} if $\big(V,\cdot\big)$ is an associative algebra (not necessarily commutative), $\big(V,\{-,-\}\big)$ is a left pre-Lie algebra, and one has
    \begin{align}
        &\{ab,c\}=\{ba,c\},\\[7pt]
        &\{a,bc\}=\{a,b\}c+b\{a,c\}.
    \end{align}
\end{definition}

\begin{example}
\begin{enumerate}[label=\theenumi),leftmargin=3ex]
    \item Let $A$ be a commutative algebra with a derivation $\partial$. Then $A$ is a Poisson-left-pre-Lie algebra for $\{a,b\}:=a\partial(b)$.
    \vspace{5pt}

    \item Let us consider an example of the form 1) but with a noncommutative $A$. Let $U$ be the associative algebra of strictly upper-triangular matrices of size $n\times n$. It is graded by the distance to the main diagonal, i.e., $U_k:=\operatorname{span}_{\Bbbk}(E_{i,i+k}\mid i=1,\ldots,n-k)$, $U_kU_l\subset U_{k+l}$ with $U_n=0$. Take $x\in U_{n-3}$. Then $U$ is a Poisson-left-pre-Lie algebra for $\{a,b\}:=a[x,b]=a(xb-bx)$. 

    \item Any associative algebra $A$ with an associative multiplication $m$ such that any triple product is zero, i.e., $m(m(a,b),c)=0$ for any $a,b,c\in A$, is trivially a Poisson-left-pre-Lie algebra for $\cdot=\{-,-\}=m$. This includes for example two-step nilpotent Lie algebras.
    \vspace{5pt}

    \item Let $A=\kk[x_1,\ldots,x_n]$ be the algebra of commutative polynomials in $n$ variables. Then the left $A$-module of derivations of $A$, denoted by $\operatorname{Der}(A)$, is isomorphic to $A^{\oplus n}$, where the $i$-th component corresponds to the partial derivative $\partial_i=\frac{\partial}{\partial x_i}$. Recall that $\operatorname{Der}(A)$ is a left pre-Lie algebra with the product
    \begin{align}
        \{f\partial,g\delta\}:=f\partial(g)\delta,\hspace{20pt}\partial,\delta\in\{\partial_1,\ldots,\partial_n\},
    \end{align}
    see, e.g., Section 3.1 in \cite{manchon2011short}. Equip $\Der(A)$ with an associative product via the $A$-module isomorphism $\operatorname{Der}(A)\simeq A^{\oplus n}$, i.e., set
    \begin{align}
        (f\partial_i)\cdot (g\partial_j):=\delta_{i,j}fg\partial_j,
    \end{align}
    where $\delta_{i,j}$ is the Kronecker delta. Then $\operatorname{Der}(A)$ equipped with $\cdot$ and $\{-,-\}$ is a Poisson-left-pre-Lie algebra.

\end{enumerate}
\end{example}

\begin{definition}
    We say that a pair of coupled double Poisson brackets $\br_{\id}, \br_{(12)}:\T(V)\otimes \T(V)\longrightarrow \T(V)\otimes \T(V)$ is a pair of \textit{linear} coupled double Poisson brackets if each of the brackets is linear in the sense of Definition \ref{def1}.
\end{definition}

\begin{theorem}\label{thm_linear}
    Any pair of linear coupled double Poisson brackets on $\T(V)$ is of the form
    \begin{align}
        &\{\!\!\{a,b\}\!\!\}_{\id}=1\otimes \{a,b\}-\{b,a\}\otimes 1, \hspace{-100pt}&a,b\in V,\\[7pt]
        &\{\!\!\{a,b\}\!\!\}_{(12)}=1\otimes ab-ba\otimes 1, &a,b\in V
    \end{align}
    for a Poisson-left-pre-Lie algebra $\big(V,\cdot,\{-,-\}\big)$.
\end{theorem}
\begin{proof}
    Due to Proposition \ref{prop4} and Corollary \ref{cor2}, we know that $\br_{\id}$ is a right double Poisson bracket and $\br_{(12)}$ is a double Poisson bracket. Now, we need to check the compatibility condition \eqref{f2}. It is enough to check it on generators $V\subset \T(V)$ due to Corollary \ref{cor1}. 
    
    Let $a,b,c\in V$. Then one has
    \begin{align}
       \Big\{\!\!\!\Big\{a,\{\!\!\{b,c\}\!\!\}_{\id}\Big\}\!\!\!\Big\}_{(12)}^{\text{left}}=-\{\!\!\{a,\{c,b\}\}\!\!\}_{(12)}\otimes 1=-1\otimes a\{c,b\}\otimes 1+\{c,b\}a\otimes 1\otimes 1
    \end{align}

    \begin{align}
        (13)\Big\{\!\!\!\Big\{b,\{\!\!\{c,a\}\!\!\}_{\id}\Big\}\!\!\!\Big\}_{(12)}^{\text{right}}=(13)\Big[1\otimes \{\!\!\{b,\{c,a\}\}\!\!\}_{(12)}\Big]&=(13)\Big[1\otimes 1\otimes b\{c,a\}-1\otimes \{c,a\}b\otimes 1\Big]\\
        &=b\{c,a\}\otimes 1\otimes 1-1\otimes \{c,a\}b\otimes 1
    \end{align}

    \begin{align}
        (132)\Big\{\!\!\!\Big\{c,\{\!\!\{a,b\}\!\!\}_{(12)}\Big\}\!\!\!\Big\}_{\id}^{\text{left}}=-(132)\Big[\{\!\!\{c,ba\}\!\!\}_{\id}\otimes 1\Big]&=-(132)\Big[1\otimes \{c,ba\}\otimes 1-\{ba,c\}\otimes 1\otimes 1\Big]\\
        &=-\{c,ba\}\otimes 1\otimes 1+1\otimes 1\otimes \{ba,c\}
    \end{align}

    \begin{align}
        (23)\Big\{\!\!\!\Big\{c,\{\!\!\{a,b\}\!\!\}_{(12)}\Big\}\!\!\!\Big\}_{\id}^{\text{right}}=(23)\Big[1\otimes \{\!\!\{c,ab\}\!\!\}_{\id}\Big]&=(23)\Big[1\otimes 1\otimes \{c,ab\}-1\otimes \{ab,c\}\otimes 1\Big]\\
        &=1\otimes \{c,ab\}\otimes 1-1\otimes 1\otimes \{ab,c\}
    \end{align}

    Taking sum, one obtains
    \begin{align}
        &1\otimes 1\otimes \Big[\{ba,c\}-\{ab,c\}\Big]\\
        &+1\otimes \Big[-a\{c,b\}-\{c,a\}b+\{c,ab\}\Big]\otimes 1\\
        &+\Big[\{c,b\}a+b\{c,a\}-\{c,ba\}\Big]\otimes 1\otimes 1,
    \end{align}    
    which is zero.
\end{proof}

\section{Quadratic coupled double Poisson brackets}\label{sec1}
In this section, we continue to consider only coupled double Poisson brackets from the third case of Remark \ref{rem2}, i.e., such that $\{\!\!\{a,b\}\!\!\}_{\id}, \{\!\!\{a,b\}\!\!\}_{(12)}\in A\otimes A$. 

Let $V$ be a finite dimensional vector space. Consider an operator $\mathcal{R}:V\otimes V\longrightarrow V\otimes V$ that satisfies $\mathcal R^{12}=-\mathcal R^{21}$ and the associative Yang-Baxter equation
\begin{align}\label{AYBE}\tag{AYBE}
	\mathcal R^{12}\mathcal R^{13}-\mathcal R^{23}\mathcal R^{12} +\mathcal R^{13}\mathcal R^{23}=0
\end{align}

Here $\mathcal R^{ab}$ denotes $\mathcal R$ acting on the $a$-th and $b$-th factors in $V^{\otimes 3}$, e.g., $\mathcal{R}^{21}:=(12)\mathcal{R}(12)$.

We will denote the left-hand side by $\operatorname{AYBE}_{\RR}$.

\begin{proposition}[\cite{schedler2009poisson}, \cite{odesskii2013double}]
	The assignment $u\otimes v\mapsto \RR(u,v)$ for $u,v\in V$, defines a double Poisson bracket $\dbr{-}{-}_{(12)}:\T(V)\otimes\T(V)\longrightarrow\T(V)\otimes\T(V)$ on $\T(V)$.
\end{proposition}

Consider now a linear map $r:V\otimes V\longrightarrow V\otimes V$ such that $r$ satisfies $r^{12}=-r^{21}$ and the classical Yang-Baxter equation
\begin{align}\label{CYBE}\tag{CYBE}
    [r^{12},r^{13}] + [r^{12},r^{23}] + [r^{13},r^{23}] =0
\end{align}

We will denote the left-hand side by $\operatorname{CYBE_r}$.

\begin{proposition}[Proposition 3.17 \cite{fairon2023around}]
    The assignment $u\otimes v\mapsto r(u,v)$ for $u,v\in V$ defines a right double Poisson bracket $\dbr{-}{-}_{\id}:\T(V)\otimes \T(V)\longrightarrow \T(V)\otimes \T(V)$ on $\T(V)$.
\end{proposition}
\begin{proof}
 By the proof of Theorem 2.8 in \cite{odesskii2013double}, the double Jacobiator, which is the left-hand side of the double Jacobi identity for the double Poisson bracket in the sense of Van den Bergh, equals
 \begin{align}
    \Jac(a,b,c)=-(13)\operatorname{AYBE}_r(c,b,a),\hspace{20pt} a,b,c\in V,
 \end{align}
 hence the right double Jacobiator evaluated at $a,b,c\in V$ is $-(13)\operatorname{CYBE}_r(c,b,a)$ by the skew-symmetry of $r$.
\end{proof}

\begin{proposition}\label{prop5}
    Let $\RR,r:V\otimes V\longrightarrow V\otimes V$ be skew-symmetric solutions of AYBE and CYBE respectively. If they satisfy the following compatibility condition
    \begin{align}\label{f32}
            [\RR^{12},r^{13}+r^{23}]=0,
        \end{align}
    then $\br_{\id},\br_{(12)}:\operatorname{T}(V)\otimes \T(V)\longrightarrow \T(V)\otimes \T(V)$ defined by
    \begin{align}
        &\{\!\!\{a,b\}\!\!\}_{\id}=r(a,b),\\
        &\{\!\!\{a,b\}\!\!\}_{(12)}=\RR(a,b),
    \end{align}
    is a pair of coupled double Poisson brackets.
\end{proposition}
\begin{proof}
It only remains to check that the compatibility condition above is equivalent to the condition \eqref{f2} from the third case of Remark \ref{rem2}.
\begin{align}
    \Big\{\!\!\!\Big\{a,\{\!\!\{b,c\}\!\!\}_{\id}\Big\}\!\!\!\Big\}_{(12)}^{\text{left}}&=\RR^{12}r^{23}(a,b,c),\\[10pt]
            (13)\Big\{\!\!\!\Big\{b,\{\!\!\{c,a\}\!\!\}_{\id}\Big\}\!\!\!\Big\}_{(12)}^{\text{right}}&=(13)\Big\{\!\!\!\Big\{b,r(c,a)\Big\}\!\!\!\Big\}_{(12)}^{\text{right}}=(13)\big[r'(c,a)\otimes \RR(b,r''(c,a))\big]\\[7pt]
            &=(13)(12)\big[\RR'(b,r''(c,a))\otimes r'(c,a)\otimes \RR''(b,r''(c,a))\big]=(123)\RR^{13}r^{23}(b,c,a)\\[7pt]
            &=(123)\RR^{13}r^{23}(132)(a,b,c)=\RR^{21}r^{31}(a,b,c)=\RR^{12}r^{13}(a,b,c)\\[10pt]
    (132)\Big\{\!\!\!\Big\{c,\{\!\!\{a,b\}\!\!\}_{(12)}\Big\}\!\!\!\Big\}_{\id}^{\text{left}}&=(132)\Big\{\!\!\!\Big\{c,\RR(a,b)\Big\}\!\!\!\Big\}_{\id}^{\text{left}}=(132)r^{12}\RR^{23}(123)(a,b,c)=r^{31}\RR^{12}(a,b,c)\\[7pt]
    &=-r^{13}\RR^{12}(a,b,c)\\[10pt]
            (23)\Big\{\!\!\!\Big\{c,\{\!\!\{a,b\}\!\!\}_{(12)}\Big\}\!\!\!\Big\}_{\id}^{\text{right}}&=(23)\Big\{\!\!\!\Big\{c,\RR(a,b)\Big\}\!\!\!\Big\}_{\id}^{\text{right}}=(23)\big[\RR'(a,b)\otimes r(c,\RR''(a,b))\big]=\\[7pt]
            &=(23)(12)\big[r'(c,\RR''(a,b))\otimes\RR'(a,b)\otimes  r''(c,\RR''(a,b))\big]\\[7pt]
            &=(132)r^{13}\RR^{23}(123)(a,b,c)=r^{32}\RR^{12}(a,b,c)=-r^{23}\RR^{12}(a,b,c)
        \end{align}
\end{proof}

\subsection{Examples}

\subsubsection{Quadratic coupled double Poisson brackets on \texorpdfstring{\(\kk^2\)}{k2}}

Consider \(\kk^2\) with basis \(e_1,e_2\) and denote by \(E_{ij}\) the matrix units. We will use the notation \(A\wedge B:=A\otimes B -B\otimes A\).

Set
\[
R_I \;=\; E_{11}\wedge E_{12},\qquad
R_{II} \;=\; E_{12}\wedge E_{22},\qquad 
r \;=\; (E_{11}+E_{22})\wedge E_{12}.
\]

One easily checks that \(R_I,R_{II}\) satisfy AYBE, \(r\) satisfies CYBE, and 
\[
[R_I^{12},\,r^{13}+r^{23}]=0, \qquad [R_{II}^{12},\,r^{13}+r^{23}]=0.
\]
Hence by Proposition \ref{prop5}, pairs \((r,R_I)\) and \((r,R_{II})\) define coupled double Poisson brackets on \(\kk\langle x,y\rangle\).

For \((r,R_I)\) the coupled double Poisson brackets are given by
\begin{align*}
&\{\!\!\{x,x\}\!\!\}_{\id}^{I}=0,
&&\{\!\!\{x,y\}\!\!\}_{\id}^{I}=x\otimes x,
&&\{\!\!\{y,y\}\!\!\}_{\id}^{I}=y\otimes x-x\otimes y,\\
&\{\!\!\{x,x\}\!\!\}_{(12)}^{I}=0,
&&\{\!\!\{x,y\}\!\!\}_{(12)}^{I}=x\otimes x,
&&\{\!\!\{y,y\}\!\!\}_{(12)}^{I}=0,
\end{align*}
and for \((r,R_{II})\) by
\begin{align*}
&\{\!\!\{x,x\}\!\!\}_{\id}^{II}=0,
&&\{\!\!\{x,y\}\!\!\}_{\id}^{II}=x\otimes x,
&&\{\!\!\{y,y\}\!\!\}_{\id}^{II}=y\otimes x-x\otimes y,\\
&\{\!\!\{x,x\}\!\!\}_{(12)}^{II}=0,
&&\{\!\!\{x,y\}\!\!\}_{(12)}^{II}=0,
&&\{\!\!\{y,y\}\!\!\}_{(12)}^{II}=x\otimes y-y\otimes x.
\end{align*}

\subsubsection{Quadratic coupled double Poisson brackets on \texorpdfstring{\(\kk^3\)}{k3}}

Consider \(\kk^3\) and the following elements of \(\bigwedge^2\operatorname{Mat}_3(\kk)\)
\begin{align}
R_1 &= E_{32}\wedge E_{31},\\
R_2 &= E_{11}\wedge E_{12}+E_{13}\wedge E_{12}+E_{11}\wedge E_{23}+E_{21}\wedge E_{13}+E_{22}\wedge E_{23},\\
R_3 &= E_{13}\wedge E_{12}+E_{11}\wedge E_{23}+E_{21}\wedge E_{13}+E_{22}\wedge E_{23},\\
R_4 &= E_{22}\wedge E_{23},\\
R_5 &= E_{11}\wedge E_{23}+E_{21}\wedge E_{13}+E_{22}\wedge E_{23},\\
R_6 &= E_{11}\wedge E_{13}+E_{11}\wedge E_{23}+E_{33}\wedge E_{23},\\
R_7 &= E_{13}\wedge E_{21}+E_{33}\wedge E_{23},\\
R_8 &= E_{11}\wedge E_{23}+E_{33}\wedge E_{23},
\end{align}
and their transposed-duals
\[
R_{8+i}=T^{\otimes 2}(R_i),\qquad i=1,\dots,8,
\]
where \(T\) is the transposition \(T(E_{ab})=E_{ba}\). 

\begin{theorem}[\cite{sokolov2013classification}]
    The matrices \(R_1,\ldots,R_{16}\in \bigwedge^2\operatorname{Mat}_3(\kk)\) and their scalar multiples are the only skew-symmetric solutions of AYBE on \(\kk^3\) modulo the conjugation action of \(\operatorname{GL}_3(\kk)\).
\end{theorem}

The list \textbf{A,B,C,D,E} below contains skew-symmetric solutions \(R\) of AYBE and skew-symmetric solutions \(r\) of CYBE which are compatible in the sense that \([R^{12},\,r^{13}+r^{23}]=0\). The corresponding quadratic coupled double Poisson brackets on \(\kk\langle x,y,z\rangle\) are also presented. Conjecturally, this list is exhaustive.

\begin{conjecture}
    The pairs of matrices \((R,r)\) listed in \textbf{A,B,C,D,E} below together with their transposed duals \((T^{\otimes 2}(R),T^{\otimes 2}(r))\) and scalar multiples thereof are the only skew-symmetric solutions of AYBE and CYBE on \(\kk^3\) which are compatible in the sense that \([R^{12},\,r^{13}+r^{23}]=0\) modulo the conjugation action of \(\operatorname{GL}_3(\kk)\).
\end{conjecture}

\textbf{Case A.}
\begin{align}
R_A&=R_1=E_{32}\wedge E_{31},\\[7pt]
r_{A.1}
&=a_1\!\left(-2\,E_{11}\wedge E_{22}-E_{11}\wedge E_{33}+E_{22}\wedge E_{33}\right)\\
&\quad +(E_{11}+E_{22}+E_{33})\wedge\left(a_2E_{12}+a_3E_{21}+a_4E_{31}+a_5E_{32}\right)
+a_6\,E_{31}\wedge E_{32},\\
r_{A.2}&=\bigl(a_7(E_{11}+E_{22}+E_{33})+a_8(E_{11}-E_{22})+a_9E_{32}+a_{10}E_{21}\bigr)\wedge E_{31},\\
r_{A.3}&=
(E_{11}+E_{22}+E_{33})\wedge (a_{11}E_{12}+a_{12}E_{32}+a_{13}E_{31})
+(a_{14}\,E_{12}+a_{15}\,E_{31})\wedge E_{32},\\
r_{A.4}
&=
a_{16}\Bigl((E_{11}-E_{22})\wedge\Bigl(4E_{12}+E_{21}\Bigr)
+4E_{12}\wedge E_{21}\Bigr)\\
&\quad
+a_{17}\,(E_{11}+E_{22}+E_{33})\wedge\Bigl(-2(E_{11}-E_{22})-4E_{12}+E_{21}\Bigr)
+a_{18}\,E_{31}\wedge E_{32},\\
r_{A.5}
&=
a_{19}(E_{11}+E_{22}+E_{33})\wedge(E_{11}-E_{22}-E_{12}+E_{21}+E_{31}-E_{32})
\\
&\quad
+a_{20}\,(E_{11}-E_{22}-E_{12}+E_{21})\wedge(E_{31}-E_{32})
+a_{21}\,E_{31}\wedge E_{32},\\
r_{A.6}
&=
a_{22}\,(E_{11}+E_{22}+E_{33})\wedge E_{31}
+a_{23}\Bigl((E_{11}-E_{22})\wedge E_{31}+E_{21}\wedge E_{32}\Bigr),
\end{align}
where \(a_1,\ldots,a_{23}\in\kk\).

The corresponding \(6\) pairs of coupled double Poisson brackets on \(\kk\langle x,y,z\rangle\)
\[
\Bigl(\br_{\id}^{A.1},\br_{(12)}^{A}\Bigr),\qquad
\Bigl(\br_{\id}^{A.2},\br_{(12)}^{A}\Bigr),
\]
\[
\Bigl(\br_{\id}^{A.3},\br_{(12)}^{A}\Bigr),\qquad
\Bigl(\br_{\id}^{A.4},\br_{(12)}^{A}\Bigr),
\]
\[
\Bigl(\br_{\id}^{A.5},\br_{(12)}^{A}\Bigr),\qquad
\Bigl(\br_{\id}^{A.6},\br_{(12)}^{A}\Bigr)
\]
are given by (only nonzero values are displayed)
\begin{align*}
\{\!\!\{x,y\}\!\!\}_{(12)}^{A} &= - z\otimes z
\end{align*}

\begin{align*}
\{\!\!\{x,x\}\!\!\}_{\id}^{A.1}
&= a_{3}\,(x\otimes y-y\otimes x)+a_{4}\,(x\otimes z-z\otimes x),\\
\{\!\!\{x,y\}\!\!\}_{\id}^{A.1}
&= a_{2}\,x\otimes x-2a_{1}\,x\otimes y+a_{5}\,x\otimes z-a_{3}\,y\otimes y-a_{4}\,z\otimes y+a_{6}\,z\otimes z,\\
\{\!\!\{x,z\}\!\!\}_{\id}^{A.1}
&= -a_{1}\,x\otimes z-a_{3}\,y\otimes z-a_{4}\,z\otimes z,\\
\{\!\!\{y,y\}\!\!\}_{\id}^{A.1}
&= a_{2}\,(y\otimes x-x\otimes y)+a_{5}\,(y\otimes z-z\otimes y),\\
\{\!\!\{y,z\}\!\!\}_{\id}^{A.1}
&= -a_{2}\,x\otimes z+a_{1}\,y\otimes z-a_{5}\,z\otimes z.
\end{align*}

\begin{align*}
\{\!\!\{x,x\}\!\!\}_{\id}^{A.2}
&= (a_{7}+a_{8})\,(x\otimes z-z\otimes x)
+a_{10}\,(y\otimes z-z\otimes y),\\
\{\!\!\{x,y\}\!\!\}_{\id}^{A.2}
&= -(a_{7}-a_{8})\,z\otimes y-a_{9}\,z\otimes z,\\
\{\!\!\{x,z\}\!\!\}_{\id}^{A.2}
&= -a_{7}\,z\otimes z.
\end{align*}

\begin{align*}
\{\!\!\{x,x\}\!\!\}_{\id}^{A.3}
&= a_{13}\,(x\otimes z - z\otimes x),\\
\{\!\!\{x,y\}\!\!\}_{\id}^{A.3}
&= a_{11}\,x\otimes x + a_{12}\,x\otimes z - a_{13}\,z\otimes y + a_{15}\,z\otimes z,\\
\{\!\!\{x,z\}\!\!\}_{\id}^{A.3}
&= -a_{13}\,z\otimes z,\\
\{\!\!\{y,y\}\!\!\}_{\id}^{A.3}
&= a_{11}\,(y\otimes x - x\otimes y)
+a_{14}\,(x\otimes z - z\otimes x)
+a_{12}\,(y\otimes z - z\otimes y),\\
\{\!\!\{y,z\}\!\!\}_{\id}^{A.3}
&= -a_{11}\,x\otimes z - a_{12}\,z\otimes z.
\end{align*}

\begin{align*}
\{\!\!\{x,x\}\!\!\}_{\id}^{A.4}
&= \left(a_{16}+a_{17}\right)\,(x\otimes y-y\otimes x),\\[4pt]
\{\!\!\{x,y\}\!\!\}_{\id}^{A.4}
&= 4(a_{16}-a_{17})\,x\otimes x
+4a_{17}\,x\otimes y
-4a_{16}\,y\otimes x
+\left(a_{16}-a_{17}\right)\,y\otimes y
+a_{18}\,z\otimes z,\\
\{\!\!\{x,z\}\!\!\}_{\id}^{A.4}
&= a_{17}\,(2x-y)\otimes z,\\
\{\!\!\{y,y\}\!\!\}_{\id}^{A.4}
&= 4(a_{16}+a_{17})\,(x\otimes y-y\otimes x),\\
\{\!\!\{y,z\}\!\!\}_{\id}^{A.4}
&= 2a_{17}\,(2x-y)\otimes z.
\end{align*}

\begin{align*}
\{\!\!\{x,x\}\!\!\}_{\id}^{A.5}
&= a_{19}\,(x\otimes y-y\otimes x)
+(a_{19}+a_{20})\,(x\otimes z-z\otimes x)
+a_{20}\,(y\otimes z-z\otimes y),\\
\{\!\!\{x,y\}\!\!\}_{\id}^{A.5}
&= -a_{19}\,x\otimes x-a_{19}\,y\otimes y+a_{21}\,z\otimes z
-2a_{19}\,x\otimes y-(a_{19}+a_{20})\,x\otimes z+a_{20}\,z\otimes x\\
&\quad -a_{20}\,y\otimes z-(a_{19}-a_{20})\,z\otimes y,\\
\{\!\!\{x,z\}\!\!\}_{\id}^{A.5}
&= -a_{19}\,(x\otimes z+y\otimes z+z\otimes z),\\
\{\!\!\{y,y\}\!\!\}_{\id}^{A.5}
&= a_{19}\,(x\otimes y-y\otimes x)
+a_{20}\,(x\otimes z-z\otimes x)
-(a_{19}-a_{20})\,(y\otimes z-z\otimes y),\\
\{\!\!\{y,z\}\!\!\}_{\id}^{A.5}
&= a_{19}\,(x\otimes z+y\otimes z+z\otimes z).
\end{align*}

\begin{align*}
\{\!\!\{x,x\}\!\!\}_{\id}^{A.6}
&= (a_{22}+a_{23})\,(x\otimes z - z\otimes x),\\
\{\!\!\{x,y\}\!\!\}_{\id}^{A.6}
&= a_{23}\,y\otimes z -(a_{22}-a_{23})\,z\otimes y,\\
\{\!\!\{x,z\}\!\!\}_{\id}^{A.6}
&= -a_{22}\,z\otimes z.
\end{align*}

\textbf{Case B.}
\begin{align}
R_B&=R_2=
E_{11}\wedge E_{12}
+E_{13}\wedge E_{12}
+E_{11}\wedge E_{23}
+E_{21}\wedge E_{13}
+E_{22}\wedge E_{23},\\[7pt]
r_B
&=
E_{11}\wedge E_{12}
+2E_{11}\wedge E_{23}
-E_{12}\wedge (E_{22}+E_{33})
+2E_{22}\wedge E_{23}
-2E_{23}\wedge E_{33}.
\end{align}

The corresponding pair \(\Bigl(\br_{\id}^{B},\br_{(12)}^{B}\Bigr)\) of coupled double Poisson brackets on \(\kk\langle x,y,z\rangle\) is given by (only nonzero values are displayed)
\begin{align}
\{\!\!\{x,y\}\!\!\}_{(12)}^{B}&= x\otimes x, &
\{\!\!\{x,z\}\!\!\}_{(12)}^{B}&= x\otimes y + y\otimes x, &
\{\!\!\{y,z\}\!\!\}_{(12)}^{B}&= -\,x\otimes x + y\otimes y,
\end{align}

\begin{align}
\{\!\!\{x,y\}\!\!\}_{\id}^{B} &= x\otimes x, &
\{\!\!\{x,z\}\!\!\}_{\id}^{B} &= 2\,x\otimes y, &
\{\!\!\{y,y\}\!\!\}_{\id}^{B} &= -\,x\otimes y + y\otimes x,\\
\{\!\!\{y,z\}\!\!\}_{\id}^{B} &= -\,x\otimes z + 2\,y\otimes y, &
\{\!\!\{z,z\}\!\!\}_{\id}^{B} &= -\,2\,y\otimes z + 2\,z\otimes y.
\end{align}

\textbf{Case C.}
\begin{align}
R_C
&=
E_{13}\wedge E_{12}
+E_{11}\wedge E_{23}
+E_{21}\wedge E_{13}
+E_{22}\wedge E_{23},\\[7pt]
r_C
&=
c_1\bigl(E_{11}\wedge E_{13}-E_{13}\wedge E_{22}-E_{13}\wedge E_{33}\bigr)\nonumber\\
&\qquad
+c_2\bigl(E_{11}\wedge E_{23}+E_{22}\wedge E_{23}-E_{23}\wedge E_{33}\bigr)
+c_3\,E_{13}\wedge E_{23},
\end{align}
where \(c_1,c_2,c_3\in\kk\).

The corresponding pair \(\Bigl(\br_{\id}^{C},\br_{(12)}^{C}\Bigr)\) of coupled double Poisson brackets on \(\kk\langle x,y,z\rangle\) is given by (only nonzero values are displayed)
\begin{align}
\{\!\!\{x,z\}\!\!\}_{(12)}^{C} &= x\otimes y + y\otimes x,
\phantom{\{\!\!\{y,z\}\!\!\}_{\id}^{C}}
\{\!\!\{y,z\}\!\!\}_{(12)}^{C} = -\,x\otimes x + y\otimes y,
\end{align}

\begin{align}
\{\!\!\{x,z\}\!\!\}_{\id}^{C} &= c_1\,x\otimes x + c_2\,x\otimes y,
\qquad
\{\!\!\{y,z\}\!\!\}_{\id}^{C} = c_1\,y\otimes x + c_2\,y\otimes y,\\[6pt]
\{\!\!\{z,z\}\!\!\}_{\id}^{C}
&= c_1\bigl(-\,x\otimes z + z\otimes x\bigr)
 +c_2\bigl(-\,y\otimes z + z\otimes y\bigr)
 +c_3\bigl(x\otimes y - y\otimes x\bigr).
\end{align}

\textbf{Case D.}
\begin{align}
R_{D.1}&=R_4=E_{22}\wedge E_{23},\\
R_{D.2}&=R_5=E_{11}\wedge E_{23}+E_{21}\wedge E_{13}+E_{22}\wedge E_{23},\\
R_{D.3}&=R_7=
E_{13}\wedge E_{21}+E_{33}\wedge E_{23},\\
R_{D.4}&=R_8=(E_{11}+E_{33})\wedge E_{23},\\[7pt]
r_{D.1}&=
d_1\bigl(E_{11}\wedge E_{22}+E_{11}\wedge E_{33}\bigr)
+d_2\,E_{11}\wedge E_{23}
+d_3\bigl(E_{22}\wedge E_{23}-E_{23}\wedge E_{33}\bigr),\\
r_{D.2}&=
d_4\bigl(E_{11}\wedge E_{23}+E_{22}\wedge E_{23}-E_{23}\wedge E_{33}\bigr)+d_5\,E_{11}\wedge E_{13}\\
&\hspace{60pt}
+d_6\bigl(E_{13}\wedge E_{22}+E_{13}\wedge E_{33}\bigr)
+d_7\,E_{13}\wedge E_{23},
\end{align}
where \(d_1,\ldots,d_7\in\kk\).

The corresponding \(8\) pairs of coupled double Poisson brackets on \(\kk\langle x,y,z\rangle\)
\begin{align}
\Bigl(\br_{\id}^{D.1},\br_{(12)}^{D.1}\Bigr),\qquad
\Bigl(\br_{\id}^{D.1},\br_{(12)}^{D.2}\Bigr),\nonumber\\
\Bigl(\br_{\id}^{D.1},\br_{(12)}^{D.3}\Bigr),\qquad
\Bigl(\br_{\id}^{D.1},\br_{(12)}^{D.4}\Bigr).
\end{align}

\begin{align}
\Bigl(\br_{\id}^{D.2},\br_{(12)}^{D.1}\Bigr),\qquad
\Bigl(\br_{\id}^{D.2},\br_{(12)}^{D.2}\Bigr),\nonumber\\
\Bigl(\br_{\id}^{D.2},\br_{(12)}^{D.3}\Bigr),\qquad
\Bigl(\br_{\id}^{D.2},\br_{(12)}^{D.4}\Bigr).
\end{align}
are given by (only nonzero values are displayed)
\begin{align}
\{\!\!\{y,z\}\!\!\}_{(12)}^{D.1}&=y\otimes y.
\end{align}

\begin{align}
\{\!\!\{x,z\}\!\!\}_{(12)}^{D.2}&=x\otimes y+y\otimes x,\qquad \{\!\!\{y,z\}\!\!\}_{(12)}^{D.2}=y\otimes y.
\end{align}

\begin{align}
\{\!\!\{x,z\}\!\!\}_{(12)}^{D.3}&=-\,y\otimes x,\qquad \{\!\!\{z,z\}\!\!\}_{(12)}^{D.3}=z\otimes y-y\otimes z.
\end{align}

\begin{align}
\{\!\!\{x,z\}\!\!\}_{(12)}^{D.4}&=x\otimes y,\qquad \{\!\!\{z,z\}\!\!\}_{(12)}^{D.4}=z\otimes y-y\otimes z.
\end{align}

\begin{align}
\{\!\!\{x,y\}\!\!\}_{\id}^{D.1}&=d_1\,x\otimes y,\qquad \{\!\!\{x,z\}\!\!\}_{\id}^{D.1}=d_1\,x\otimes z+d_2\,x\otimes y,\\
\{\!\!\{y,z\}\!\!\}_{\id}^{D.1}&=d_3\,y\otimes y,\qquad \{\!\!\{z,z\}\!\!\}_{\id}^{D.1}=d_3\,(z\otimes y-y\otimes z).
\end{align}

\begin{align}
\{\!\!\{x,z\}\!\!\}_{\id}^{D.2}&=d_5\,x\otimes x+d_4\,x\otimes y,\qquad 
\{\!\!\{y,z\}\!\!\}_{\id}^{D.2}=d_4\,y\otimes y-d_6\,y\otimes x,\\
\{\!\!\{z,z\}\!\!\}_{\id}^{D.2}&=d_4\,(z\otimes y-y\otimes z)+d_6\,(x\otimes z-z\otimes x)+d_7\,(x\otimes y-y\otimes x).
\end{align}

\textbf{Case E.}
\begin{align}
R_E&=R_6=
E_{11}\wedge E_{13}+E_{11}\wedge E_{23}+E_{33}\wedge E_{23},\\[7pt]
r_E
&=
e_1\bigl(E_{11}\wedge E_{13}-E_{13}\wedge E_{22}-E_{13}\wedge E_{33}\bigr)\nonumber\\
&\hspace{100pt}
+e_2\bigl(E_{11}\wedge E_{23}+E_{22}\wedge E_{23}-E_{23}\wedge E_{33}\bigr)
+e_3\,E_{13}\wedge E_{23},
\end{align}
where \(e_1,e_2,e_3\in\kk\).

The corresponding pair \(\Bigl(\br_{\id}^{E},\br_{(12)}^{E}\Bigr)\) of coupled double Poisson brackets on \(\kk\langle x,y,z\rangle\) is given by (only nonzero values are displayed)
\begin{align*}
\{\!\!\{x,z\}\!\!\}_{(12)}^{E}
&=x\otimes x+x\otimes y,\\
\{\!\!\{z,z\}\!\!\}_{(12)}^{E}
&=z\otimes y-y\otimes z,
\end{align*}

\begin{align*}
\{\!\!\{x,z\}\!\!\}_{\id}^{E}
&=e_1\,x\otimes x+e_2\,x\otimes y,\\
\{\!\!\{y,z\}\!\!\}_{\id}^{E}
&=e_1\,y\otimes x+e_2\,y\otimes y,\\
\{\!\!\{z,z\}\!\!\}_{\id}^{E}
&=-e_1\,x\otimes z+e_1\,z\otimes x-e_2\,y\otimes z+e_2\,z\otimes y+e_3\,x\otimes y-e_3\,y\otimes x.
\end{align*}

\subsubsection{A multiparameter family on \texorpdfstring{\(\kk^n\)}{kn}}

Let \(\lambda_1,\ldots,\lambda_n\in \kk\) be pairwise distinct elements: \(\lambda_i\neq\lambda_j\) with \(i\neq j\). 

Define a double bracket \(\br_{(12)}:\kk\langle x_1,\ldots,x_n\rangle^{\otimes 2}\longrightarrow \kk\langle x_1,\ldots,x_n\rangle^{\otimes 2}\) by
\[
\{\!\!\{x_i,x_j\}\!\!\}_{(12)}
:=
\begin{cases}
\dfrac{1}{\lambda_i-\lambda_j}
\bigl(x_i\otimes x_j+x_j\otimes x_i-x_i\otimes x_i-x_j\otimes x_j\bigr),
& i\neq j,\\[1.2ex]
0, & i=j.
\end{cases}
\]
and a right double bracket \(\br_{\id}:\kk\langle x_1,\ldots,x_n\rangle^{\otimes 2}\longrightarrow \kk\langle x_1,\ldots,x_n\rangle^{\otimes 2}\) by
\[
\{\!\!\{x_i,x_i\}\!\!\}_{\id}
=\sum_{k\neq i}\frac{1}{\lambda_i-\lambda_k}\,\bigl(x_k\otimes x_i-x_i\otimes x_k\bigr)\qquad \text{for}\ i=1,\ldots,n;
\]
and by
\[
\{\!\!\{x_i,x_j\}\!\!\}_{\id}
=\sum_{k\neq i}\frac{1}{\lambda_i-\lambda_k}\,x_k\otimes x_j
-\sum_{k\neq j}\frac{1}{\lambda_j-\lambda_k}\,x_i\otimes x_k
+\left(\sum_{k\neq j}\frac{1}{\lambda_j-\lambda_k}-\sum_{k\neq i}\frac{1}{\lambda_i-\lambda_k}\right)x_i\otimes x_j
\]
for \(i\neq j\).

It is known that \(\br_{(12)}\) is a double Poisson bracket, see Example 1 in Section 2 in \cite{odesskii2013double}. 

\begin{proposition}
    The pair \(\Bigl(\br_{\id},\br_{(12)}\Bigr)\) is a pair of coupled double Poisson brackets.
\end{proposition}
\begin{proof}
    
Let \(V=\kk^n\) with basis \(\{e_i\}_{i=1}^n\). Define \(R:V\otimes V\longrightarrow V\otimes V\) by
\[
R(e_i\otimes e_j)=
\begin{cases}
\dfrac{1}{\lambda_i-\lambda_j}
\bigl(e_i\otimes e_j+e_j\otimes e_i-e_i\otimes e_i-e_j\otimes e_j\bigr), & i\neq j,\\[1ex]
0, & i=j
\end{cases}
\]
and \(X\in \operatorname{End}(V)\) by
\[
X_{ij}=
\begin{cases}
\dfrac{1}{\lambda_i-\lambda_j}, & i\neq j,\\[1ex]
-\displaystyle\sum_{k\neq i}\frac{1}{\lambda_i-\lambda_k}, & i=j.
\end{cases}
\]
Equivalently,
\[
Xe_i=\sum_{k\neq i}\frac{1}{\lambda_i-\lambda_k}e_k
-\left(\sum_{k\neq i}\frac{1}{\lambda_i-\lambda_k}\right)e_i.
\]
Now set
\[
r:=X\otimes 1-1\otimes X.
\]

Identifying \(\operatorname{T}(V)\) with \(\kk\langle x_1,\ldots,x_n\rangle\) under \(x_i\leftrightarrow e_i\), one has
\[
\{\!\!\{a,b\}\!\!\}_{\id}=r(a,b),\hspace{30pt}\{\!\!\{a,b\}\!\!\}_{(12)}=R(a,b),\hspace{20pt} a,b\in V.
\]

Due to Proposition \ref{prop5}, it remains to check that
\[
r^{21}=-r,\qquad
[r^{12},r^{13}]+[r^{12},r^{23}]+[r^{13},r^{23}]=0,\qquad
[R^{12},\,r^{13}+r^{23}]=0.
\]

It is clear that \(r^{21}=-r\) from the very definition of \(r\). Write \(X_1:=X\otimes 1\otimes 1\), \(X_2:=1\otimes X\otimes 1\), \(X_3:=1\otimes 1\otimes X\). On \(V^{\otimes3}\), one has
\[
r^{12}=X_1-X_2,\qquad r^{13}=X_1-X_3,\qquad r^{23}=X_2-X_3.
\]
Since operators on different legs commute, \([X_a,X_b]=0\) for \(a\neq b\). Hence \( [r^{12},r^{13}]=[r^{12},r^{23}]=[r^{13},r^{23}]=0.\)

Let us prove the only remaining identity. We have
\(
r^{13}+r^{23}=X_1+X_2-2X_3.
\)
Since \(R^{12}\) acts only on legs \(1,2\), one has \([R^{12},X_3]=0\). So it is enough to prove
\(
[R^{12},X_1+X_2]=0.
\)
On \(V\otimes V\), denote \(S:=X\otimes 1+1\otimes X\), so we need to check \([R,S]=0\).

Define
\[
F_{ij}:=e_i\otimes e_j+e_j\otimes e_i-e_i\otimes e_i-e_j\otimes e_j.
\]
Then \(F_{ij}=F_{ji}\), \(F_{ii}=0\), and
\[
R(e_i\otimes e_j)=a_{ij}F_{ij}
\]
for
\[
a_{ij}:=\frac1{\lambda_i-\lambda_j}\quad (\text{for}\ i\neq j),\qquad a_{ii}:=0,
\]
Also write \(Xe_i=\sum_k b_{ik}e_k\), where
\[
b_{ik}=a_{ik}\ (\text{for}\ k\neq i),\qquad b_{ii}=-\sum_{k\neq i}a_{ik},
\]
so \(\sum_k b_{ik}=0\).

\emph{Case \(i=j\).}
Since \(R(e_i\otimes e_i)=0\), \(SR(e_i\otimes e_i)=0\). Also
\[
RS(e_i\otimes e_i)
=\sum_{k\neq i}b_{ik}R(e_k\otimes e_i)+\sum_{k\neq i}b_{ik}R(e_i\otimes e_k)=\sum_{k\neq i}b_{ik}a_{ki}F_{ki}+\sum_{k\neq i}b_{ik}a_{ik}F_{ik}=0
\]
because \(a_{ki}=-a_{ik}\) and \(F_{ki}=F_{ik}\). Hence \([R,S](e_i\otimes e_i)=0\).

\emph{Case \(i\neq j\).}
We compute \(RS(e_i\otimes e_j)\) first. Since
\[
S(e_i\otimes e_j)
=\sum_{p} b_{ip}\,e_p\otimes e_j+\sum_{q} b_{jq}\,e_i\otimes e_q,
\]
we get
\begin{align}
    RS(e_i\otimes e_j)
&=\sum_{p} b_{ip}\,R(e_p\otimes e_j)+\sum_{q} b_{jq}\,R(e_i\otimes e_q)\\
&=a_{ij}(b_{ii}+b_{jj})F_{ij}
+\sum_{k\neq i,j} b_{ik}a_{kj}F_{kj}
+\sum_{k\neq i,j} b_{jk}a_{ik}F_{ik},\label{f41}
\end{align}
because the \(p=j\) and \(q=i\) terms vanish.

Now compute \(SR(e_i\otimes e_j)\). Since \(R(e_i\otimes e_j)=a_{ij}F_{ij}\),
\[
SR(e_i\otimes e_j)=a_{ij}\,S(F_{ij}).
\]
One has
\begin{align*}
S(F_{ij})
&=S(e_i\otimes e_j)+S(e_j\otimes e_i)-S(e_i\otimes e_i)-S(e_j\otimes e_j)\\
&=\sum_k b_{ik}e_k\otimes e_j+\sum_k b_{jk}e_i\otimes e_k
  +\sum_k b_{jk}e_k\otimes e_i+\sum_k b_{ik}e_j\otimes e_k\\
&\qquad-\sum_k b_{ik}e_k\otimes e_i-\sum_k b_{ik}e_i\otimes e_k
      -\sum_k b_{jk}e_k\otimes e_j-\sum_k b_{jk}e_j\otimes e_k\\
&=\sum_k (b_{ik}-b_{jk})
\bigl(e_k\otimes e_j+e_j\otimes e_k-e_k\otimes e_i-e_i\otimes e_k\bigr).
\end{align*}
Now
\[
e_k\otimes e_j+e_j\otimes e_k-e_k\otimes e_i-e_i\otimes e_k
=(F_{kj}-F_{ki})+(e_j\otimes e_j-e_i\otimes e_i),
\]
so
\[
S(F_{ij})
=\sum_k (b_{ik}-b_{jk})(F_{kj}-F_{ki})
+\Bigl(\sum_k(b_{ik}-b_{jk})\Bigr)(e_j\otimes e_j-e_i\otimes e_i).
\]
By construction of \(b_{aa}\), \(\sum_k b_{ik}=\sum_k b_{jk}=0\), therefore
\[
S(F_{ij})=\sum_k (b_{ik}-b_{jk})(F_{kj}-F_{ki}).
\]
Finally split \(k=i,j\) and \(k\neq i,j\):
\begin{align*}
S(F_{ij})
&=(b_{ii}-b_{ji})(F_{ij}-F_{ii})+(b_{ij}-b_{jj})(F_{jj}-F_{ji})
 +\sum_{k\neq i,j}(b_{ik}-b_{jk})(F_{kj}-F_{ki})\\
&=(b_{ii}-b_{ji}-b_{ij}+b_{jj})F_{ij}
 +\sum_{k\neq i,j}(b_{ik}-b_{jk})(F_{kj}-F_{ki}).
\end{align*}
For \(i\neq j\), \(b_{ji}=a_{ji}=- a_{ij}=-b_{ij}\), so \(b_{ii}-b_{ji}-b_{ij}+b_{jj}=b_{ii}+b_{jj}.\)
Therefore
\[
S(F_{ij})
=(b_{ii}+b_{jj})F_{ij}
+\sum_{k\neq i,j}(b_{ik}-b_{jk})(F_{kj}-F_{ki}).
\]
Hence
\begin{align}
    SR(e_i\otimes e_j)
=a_{ij}(b_{ii}+b_{jj})F_{ij}+a_{ij}\sum_{k\neq i,j}(b_{ik}-b_{jk})F_{kj}
-a_{ij}\sum_{k\neq i,j}(b_{ik}-b_{jk})F_{ik},\label{f40}
\end{align}
using \(F_{ki}=F_{ik}\).

Compare coefficients of \eqref{f41} and \eqref{f40} for each \(k\neq i,j\):
\[
a_{ij}(b_{ik}-b_{jk})
= a_{ij}(a_{ik}-a_{jk})
= a_{ij}(a_{ik}+a_{kj})
= a_{ik}a_{kj}
=b_{ik}a_{kj},
\]
and
\[
-a_{ij}(b_{ik}-b_{jk})
=- a_{ij}(a_{ik}+a_{kj})
=- a_{ik}a_{kj}
= a_{ik}a_{jk}
=b_{jk}a_{ik}.
\]
So the coefficients of \(F_{ij},F_{kj},F_{ik}\) in \(SR(e_i\otimes e_j)\) and
\(RS(e_i\otimes e_j)\) are identical. Therefore
\[
SR(e_i\otimes e_j)=RS(e_i\otimes e_j)\qquad(\text{for}\ i\neq j)
\]
and the proof is complete.
\end{proof}

    \addtocontents{toc}{\protect\setcounter{tocdepth}{0}}
    \printbibliography
    \addtocontents{toc}{\protect\setcounter{tocdepth}{3}}
    \addcontentsline{toc}{section}{References}

\end{document}